\documentclass{amsart}

\usepackage{amsmath,amsfonts,amssymb,amsthm,amscd,comment,euscript}
\usepackage{stmaryrd}
\usepackage[all]{xy}
\usepackage{graphicx}
\usepackage{enumerate} 
\usepackage[colorlinks=true]{hyperref} 
\usepackage[usenames,dvipsnames]{xcolor}

\newcommand{\too}{\longrightarrow}                                  
\newcommand{\equalby}[1]{\overset{\textrm{#1}}=}             
\newcommand{\equalbyeq}[1]{\equalby{\eqref{#1}}}             
\newcommand{\tooby}[1]{\buildrel #1\over\too}                 

\newcommand{\Z}{\mathbb{Z}}               
\newcommand{\N}{\mathbb{N}}               
\newcommand{\Q}{\mathbb{Q}}              

\newcommand{\ot}{\otimes}                 
\newcommand{\id}{\mathrm{id}} 

\newcommand{\al}{\alpha}      
\newcommand{\la}{\lambda}   
\newcommand{\RS}{\Sigma}    
\newcommand{\cl}{\Lambda}  
\newcommand{\rl}{\cl_r}         
\newcommand{\wl}{\cl_w}       
\newcommand{\sco}{sc}          
\newcommand{\adj}{ad}          
\newcommand{\tor}{\mathfrak{t}} 

\newcommand{\lbr}{[\hspace{-1.5pt}[}  
\newcommand{\rbr}{]\hspace{-1.5pt}]}  
\newcommand{\FGR}[3]{#1 \lbr #2 \rbr_{#3}} 
\newcommand{\RcF}{\FGR{R}{\cl}{F}} 
\newcommand{\aug}{\epsilon}        
\newcommand{\IF}{\mathcal{I}_F}    
\newcommand{\IFR}[1]{\mathcal{I}_{F,#1}}    

\newcommand{\QW}{Q_W}
\newcommand{\qQ}{Q}
\newcommand{\sS}{S}
\newcommand{\Qc}{Q_{R,\cl}}     
\newcommand{\lr}{[\tfrac{1}{r}]} 

\newcommand{\de}{\delta}                   
\newcommand{\De}{\Delta}                  
\newcommand{\Del}{X}                        
\newcommand{\ED}{\mathcal{D}(\cl)_F} 
\newcommand{\eD}{\aug\ED}               
\newcommand{\DcF}{\mathbf{D}_F}    
\newcommand{\tDcF}{\tilde{\mathbf{D}}_F}     
\newcommand{\kp}{\kappa}              

\newcommand{\RWcopr}{\triangle_{R[W]}} 
\newcommand{\RWcoun}{\varepsilon_{R[W]}} 

\newcommand{\QWprod}{\odot} 

\newcommand{\QWRcopr}{\triangle_{\QW}} 
\newcommand{\incQtoQW}{\upsilon_\qQ} 
\newcommand{\incRWtoQW}{\upsilon_{R[W]}} 

\newcommand{\Qproj}{\pi} 
\newcommand{\Qinc}{\iota} 
\newcommand{\Qiso}{\psi} 
\newcommand{\Qker}{\ker(\Qproj)} 
\newcommand{\Qim}{\im(\QWRcopr)} 

\newcommand{\QWcopr}{\triangle} 
\newcommand{\QWcoun}{\varepsilon} 
\newcommand{\Dcoun}{\varepsilon} 

\newcommand{\TE}{\triangle^{\epsilon}}   

\newcommand{\Inc}{\Upsilon} 
\newcommand{\homcopr}{\chi} 
\newcommand{\EDcopr}{\triangle_{\ED}} 
\newcommand{\EDcoun}{\varepsilon_{\ED}} 

\newcommand{\ev}{\mathrm{ev}} 
\newcommand{\cmS}{c_S} 
\newcommand{\cmR}{c_R} 
\newcommand{\rhoS}{\rho_S} 
\newcommand{\rhoR}{\rho_R} 

\DeclareMathOperator{\im}{\mathrm{im}}         
\DeclareMathOperator{\Hom}{\mathrm{Hom}}
\DeclareMathOperator{\End}{\mathrm{End}}
\DeclareMathOperator{\Endc}{\mathrm{End^0}} 
\DeclareMathOperator{\Homc}{\mathrm{Hom^0}} 

\DeclareMathOperator{\hh}{\mathtt{h}}             

\theoremstyle{plain}
\newtheorem{theo}{Theorem}[section]
\newtheorem{prop}[theo]{Proposition}
\newtheorem{lem}[theo]{Lemma}
\newtheorem{cor}[theo]{Corollary}

\theoremstyle{definition}
\newtheorem{defi}[theo]{Definition}
\newtheorem{rem}[theo]{Remark}
\newtheorem{example}[theo]{Example}

\numberwithin{equation}{section}   

\begin{document}
%

\title{A coproduct structure on the formal affine Demazure algebra}

\author{Baptiste Calm\`es}
\author{Kirill Zainoulline}
\author{Changlong Zhong}

\address{Baptiste Calm\`es, Universit\'e d'Artois, Laboratoire de
  Math\'ematiques de Lens, France}
\email{baptiste.calmes@univ-artois.fr}

\address{Kirill Zainoulline, Department of Mathematics and Statistics,
University of Ottawa, Ottawa ON, K1N6N5, Canada}
\email{kirill@uottawa.ca}

\address{Changlong Zhong, 
Department of Mathematics and Statistical Sciences, 
University of Alberta, Canada}
\email{zhongusc@gmail.com}

\thanks{The first author acknowledges the support of the French Agence Nationale de la Recherche (ANR) under reference ANR-12-BL01-0005. The second author was supported by 
the NSERC Discovery grant  385795-2010, 
NSERC DAS grant 396100-2010 and the Early Researcher Award (Ontario). 
The last author appreciates the support of the Fields Institute.}

\subjclass[2010]{20C08, 14F43}

\keywords{oriented cohomology, formal group law, Demazure operator,
  nil Hecke ring, algebraic cobordism, Schubert calculus}

\maketitle

\begin{abstract}
In the present paper, we generalize the coproduct structure 
on nil Hecke rings introduced and studied by Kostant and Kumar 
to the context of an arbitrary algebraic oriented cohomology theory 
and its associated formal group law. 
We then construct an algebraic model of
the $T$-equivariant oriented cohomology of the variety of complete flags.
\end{abstract}

\setcounter{tocdepth}{3}
\tableofcontents

%

In \cite{KK}, \cite{kk} and \cite{Ku},
Kostant and Kumar introduced the language of (nil-)Hecke rings 
and described in a purely algebraic way equivariant
cohomology and $K$-theory of flag varieties.
In particular, for a given split simply connected semisimple linear
algebraic group $G$ with a maximal split torus $T$ and a Borel
subgroup $B\supset T$, they established two isomorphisms
(see \cite[Thm.~4.4]{kk}, \cite[Thm.~11.5.15]{Ku}) 
\[
S(\cl)\ot_{S(\cl)^W} S(\cl) \tooby{\simeq} H_T(G/B;\Q)
\quad\text{and}\quad
\Z[\cl] \ot_{\Z[\cl]^W} \Z[\cl] \tooby{\simeq} K_T(G/B),
\]
where $S(\cl)$ is the symmetric algebra of the group of characters
$\cl$ of $T$, 
$\Z[\cl]$ is the integral group ring, 
$S(\cl)^W$ and $\Z[\cl]^W$ denote the respective subrings of
invariants under the action of the Weyl
group $W$, while 
$H_T(G/B,\Q)$ and $K_T(G/B)$ are respectively the
$T$-equivariant cohomology
and the $K$-theory of the variety of Borel subgroups $G/B$.

More precisely, they constructed algebraic models ${\tt \cl}$ and ${\tt \Psi}$ of the mentioned equivariant theories, dual respectively to the nil-Hecke and Hecke rings, with product structure ${\tt \cl}$ and ${\tt \Psi}$ induced
by the coproduct structure on the (nil-)Hecke ring.
They also defined explicit maps inducing ring isomorphisms
\begin{equation}\label{KK_eq}
S(\cl)\ot_{S(\cl)^W} S(\cl)\tooby{\simeq} {\tt \cl}
\quad \text{and}\quad 
\Z[\cl] \ot_{\Z[\cl]^W} \Z[\cl] \tooby{\simeq} {\tt \Psi}. \tag{$\star$}
\end{equation}
They showed that all Hecke-type (Demazure,
BGG) operators on $S(\cl)$ and $\Z[\cl]$ actually sit in a bigger but
simpler ring in the so
called twisted group
algebra $\QW$ and, hence, they defined the respective coproduct to be
the restriction of a natural one given on $\QW$.

One goal of the present paper is to generalize the Kostant-Kumar construction 
of the duals ${\tt \cl}$ and ${\tt \Psi}$ and of the coproduct
for an arbitrary one-dimensional commutative formal
group law $F$ 
over a coefficient ring $R$ 
(the additive formal group law corresponds to  ${\tt \cl}$ and the multiplicative one to ${\tt \Psi}$). 
We achieve this by replacing ${\tt \cl}$ and ${\tt \Psi}$ by the dual of the formal affine Demazure algebra $\DcF^\star$, already introduced in \cite{HMSZ} as a generalization of the nil-Hecke ring, and by defining the coproduct on $\DcF$ via the natural coproduct given on the $F$-version of $\QW$.

As an application  we provide a complete
description of the algebra $\DcF$ in terms of generators and
relations (see Theorem~\ref{presentationDF_theo}), hence, generalizing
\cite[Thm.~6.14]{HMSZ} and, in particular, we show 
(see Theorem~\ref{mainthmdual_theo}): 

\medskip
\noindent
{\bf Theorem.}
If the torsion index $\tor$ of the root datum of $G$ is invertible and $2$ is not a zero
divisor in $R$, then there is an isomorphism of $R$-algebras
\[
\RcF \ot_{\RcF^W} \RcF \tooby{\simeq} \DcF^\star,
\]
where $\RcF$ is the formal group algebra introduced in \cite[\S2]{CPZ}.

\medskip

Note that specializing to the additive and multiplicative periodic formal group laws (see Example~\ref{addmult_ex}), we obtain the isomorphisms \eqref{KK_eq}. Also note that the algebra $\DcF^\star$ is defined under very weak assumptions on
the base ring $R$, e.g. $2$ is not a zero divisor in $R$, hence
allowing us to consider various formal group laws over finite fields
as well.

According to Levine-Morel \cite{levmor-book}, to any
algebraic oriented cohomology theory $\hh(-)$ one can associate a
formal group law $F$ over the coefficient ring $R=\hh(pt)$; moreover, over a base field of characteristic $0$,
there is a universal algebraic oriented cohomology theory $\Omega(-)$, 
called algebraic cobordism,
corresponding to the universal formal group law $U$ over the Lazard ring
$\mathbb{L}=\Omega(pt)$.
An associated $T$-equivariant theory, denoted by $\Omega_T(G/B)$, has been intensively
studied
during the last few years by Deshpande, Heller, Kiritchenko, Krishna and
Malagon-Lopez.
In particular, 
using comparison results and the Atiyah-Hirzebruch spectral sequence for complex cobordism 
it was shown \cite[Thm.~5.1]{KiKr} that after inverting the torsion
index the ring
$\Omega_T(G/B)$ can be identified 
with
$\mathbb{L}[\cl]_U\otimes_{\mathbb{L}[\cl]_U^W}\mathbb{L}[\cl]_U$.
We also refer to \cite{HHH} for similar results concerning
$T$-equivariant topological complex oriented theories and to
\cite{GaRa} for further discussion about the Borel model for
$\Omega_T(G/B)$.
In this context, our (Hecke-type) algebra $\DcF^\star$ can be used as an algebraic model for an algebraic oriented $T$-equivariant cohomology ring $\hh_T(G/B)$. A precise justification of this claim can be found in \cite{CZZ} and \cite{CZZ2}; in the present article we only deal with the algebraic construction, with no particular reference to geometry.

This paper is based on the techniques developed in \cite{CPZ} and especially in 
\cite{HMSZ}, which is devoted to the unification of various existing geometric approaches to
Hecke-type algebras. 
It is organized as follows: 

In section~\ref{sec:formgr}, we recall basic properties of the formal group
algebra $\RcF$ following \cite[\S2]{CPZ}.
In section~\ref{sec:rootsyst}, we discuss the regularity of elements of $\RcF$ corresponding to roots.
In section~\ref{sec:localized}, we study properties of the localization
of $\RcF$. 
In section~\ref{FormalDemazureOperators_sec},
we recall the definition of formal Demazure operators.
Section~\ref{sec:twisted} is devoted to the twisted formal group
algebra $\QW$ and the formal affine Demazure algebra $\DcF$.
There, we provide an important formula (Lemma~\ref{Deldecomp_lem})
expressing a product of Demazure elements
in terms of the canonical basis of $\QW$. We finish our preparation 
by relating the formal affine Demazure algebra to the subalgebra $\tDcF$ of $\QW$ stabilizing $\RcF$ (section~\ref{sec:invar}) and we show that they coincide in most cases.

In section~\ref{sec:present}, we describe the formal affine Demazure
algebra $\DcF$ in terms of generators and relations, hence,
generalizing \cite[Thm.~6.14]{HMSZ}. We also establish an isomorphism
(Theorem~\ref{DF=ED_theo}) between $\DcF$
and the algebra of Demazure operators on $\RcF$.
In key sections~\ref{sec:copr}, \ref{sec:copraffine} and
\ref{sec:compare}, we introduce and study the
coproduct structure on $\DcF$. For instance, we show (section~\ref{sec:compare}) that the
coproduct on $\DcF$ induces the coproduct on the cohomology ring
considered in \cite[\S7]{CPZ}. In section~\ref{sec:dualaffine}, we
apply the previously obtained results to construct the dual $\DcF^\star$ and
to prove  Theorem~\ref{mainthmdual_theo}. 

We expect that the mentioned results can be extended to root
systems of symmetrizable Kac-Moody algebras. 
However, it doesn't seem to be so straightforward: it requires adjusting arguments involving the torsion index or elements of maximal length (which do not exist in the Kac-Moody case); it will be the subject of a forthcoming paper.
Finally, we would like to mention the paper \cite{Zh} where some of our results (e.g. Theorem~\ref{presentationDF_theo})
have been given a simpler form in the special case of a normalized formal group law. 
\medskip

{\it Acknowledgments:}
The first author would like to thank Lina and Ren\'e Payet for their
hospitality; most of his work on the present article has been done
during a stay at their home.
The first and the second authors are grateful to the Banff International Research
Station for its hospitality in October 2012, which lead to a fruitful
collaboration on the subject of the paper.
All authors are also grateful to the Fields Institute for its hospitality
and support during the spring of 2013.

\section{Formal group algebras}\label{sec:formgr}

We recall definition and basic properties of
a formal group algebra, following \cite[\S2]{CPZ}. 

\medskip

Let $F$ be a one dimensional formal group law 
over a commutative (unital) ring $R$ (see \cite{Haz}, \cite[p. 4]{levmor-book}).
Given an integer $m\ge 0$ we use the notation
\[
  u +_F v =F(u,v),\quad
  m \cdot_F u :=\underbrace{u+_F \cdots +_F u}_{m \text{ times}},\quad
  \text{and }\quad
  (-m)\cdot_F u :=-_F(m\cdot_F u),
\]
where $-_F u$ denotes the formal inverse of $u$. 

Let $\cl$ be a lattice, i.e. a finitely generated free abelian group, 
and let $R[x_\cl]$ denote the polynomial ring over $R$ 
with variables indexed by $\cl$.
Let $\aug\colon R[x_\cl] \to R$ be the augmentation morphism
which maps $x_\la$ to $0$ for each $\la\in \cl$.  
Let $R\lbr x_\cl \rbr$ be the $\ker(\aug)$-adic completion of
the polynomial ring $R [x_\cl]$.  
Let $J_F$ be the closure of the ideal generated by $x_0$ and elements of the form 
$x_{\la_1+\la_2}-(x_{\la_1}+_F x_{\la_2})$ for all
$\la_1,\la_2 \in \cl$.  

Following \cite[Def.~2.4]{CPZ}, we define the \emph{formal group algebra} 
(or \emph{formal group ring}) to be the quotient 
\[
    \RcF := R\lbr x_\cl \rbr /J_F.
\]
The class of $x_\la$ in $\RcF$ will be denoted identically.  
By definition, $\RcF$ is a complete Hausdorff $R$-algebra 
with respect to the $\IF$-adic topology, 
where $\IF=\ker \aug$ with $\aug\colon \RcF \to R$ the induced
augmentation map. 
We set $\IF^{-i}=\RcF$ if $i\ge 0$.

Recall from \cite[Cor. 2.13]{CPZ} that choosing a basis
$\{\la_1,\ldots,\la_n\}$ of the lattice $\cl$ gives a canonical isomorphism 
$\RcF \simeq R\lbr x_1,\ldots,x_n\rbr$ 
sending 
\[
x_{\sum m_i
  \la_i} \mapsto m_1\cdot_F x_1 +_F \cdots +_F m_n \cdot_F x_n
\in m_1 x_1 + \cdots +m_n x_n +I^2,
\]
where $I=(x_1,\ldots,x_n)$. 

\begin{example}\label{addmult_ex} (see \cite[Ex.~2.18 and Ex.~2.19]{CPZ})

For the {\em additive} formal group law over $R$ given by $F_a(u,v)=u+v$ we have an $R$-algebra isomorphism
$\RcF \simeq S_R(\cl)^\wedge$,
where $S^i_R(\cl)$ is the $i$-th symmetric power of $\cl$ over $R$,
the completion is at the kernel of the augmentation map $x_\la
\mapsto 0$ 
and the isomorphism is induced by $x_\la \mapsto \la\in S^1_R(\cl)$. 

Consider the group ring 
$
R[\cl]:=
\big\{ \sum_j a_j e^{\la_j}  
\mid  a_j\in R,\; \la_j\in \cl \big\}$. 
Let ${\rm tr}\colon R[\cl] \to R$ be the trace map, i.e.  
a $R$-linear map sending every $e^\la$ to $1$. 
Let $R[\cl]^\wedge$ denote the completion of $R[\cl]$ at $\ker({\rm
  tr})$.

For the {\em multiplicative periodic} formal group law over $R$
given by $F_m(u,v)=u+v-\beta uv$, 
where $\beta$ is an invertible element in $R$, we have an $R$-algebra isomorphism
$\RcF \simeq R[\cl]^\wedge$ 
induced by
$x_\la \mapsto \beta^{-1}(1- e^\la)$ and 
$e^\la \mapsto (1-\beta x_\la)=(1-\beta x_{-\la})^{-1}$.

Consider the \emph{universal} formal group law $U(u,v)=u+v+a_{11}uv+\ldots$. Its coefficient ring is the Lazard
ring $\mathbb{L}$. By universality there is a canonical ring
homomorphism $\FGR{\mathbb{L}}{\cl}{U} \to S_\Z(\cl)^\wedge$ 
(resp. $\FGR{\mathbb{L}}{\cl}{U} \to \Z[\beta,\beta^{-1}][\cl]^\wedge$) given by sending all the
coefficients $a_{ij}$ of $U$ to $0$ (resp. all coefficients except $a_{11}$ to 0
and $a_{11}\mapsto -\beta$). 
\end{example}

We will need the following facts concerning
formal group algebras.
We refer to the appendix for the properties and definition of regular elements.

\begin{lem}\label{regularinject_lem}
Let $\cl \subseteq \cl'$ be two lattices. If $|\cl'/\cl|$ is regular
in $R$, 
then the natural map $\RcF \to \FGR{R}{\cl'}{F}$ is an injection. 
\end{lem}

\begin{proof}
Choose bases of $\cl$ and $\cl'$, 
and let $A$ be the matrix expressing the vectors of the first in terms
of the second. Then by Lemma~\ref{detmatrix_lem}, $\det(A)=\pm|\cl'/\cl|$. The lemma follows 
from Corollary~\ref{changevarinj_prop} through isomorphisms $\RcF \simeq R\lbr x_1,\ldots,x_n\rbr$ and $\FGR{R}{\cl'}{F}\simeq R\lbr x_1,\ldots,x_n\rbr$ associated to the bases. 
\end{proof}

\begin{lem} \label{regularirred_lem}
Let $\cl = \cl_1 \oplus \cl_2$ be a direct sum of two lattices and let
$\la \in \cl_1$. 
Then $x_\la$ is regular in $\RcF$ if and only if it is regular in $\FGR{R}{\cl_1}{F}$.
\end{lem}

\begin{proof}
By \cite[Thm.~2.1]{CPZ}, there is a natural isomorphism of
$\FGR{R}{\cl_1}{F}$-algebras $\FGR{R}{\cl_1\oplus \cl_2}{F} \simeq
\FGR{(\FGR{R}{\cl_1}{F})}{\cl_2}{F}$. The result then follows by
Corollary~\ref{regularpowerseries_lem}
applied to $f=x_\la\in R'=\FGR{R}{\cl_1}{F}$ and $\FGR{R'}{\cl_2}{F}\simeq R'\lbr x_1,\ldots, x_l\rbr$ with $l$ the rank of $\cl_2$.
\end{proof}

\section{Root systems and regular elements}\label{sec:rootsyst}

We recall several auxiliary facts concerning
root datum following \cite[Exp.~XXI]{SGA} and \cite{Bo68}. 
We provide a criteria for regularity of a generator of the formal group
algebra (see Lemma~\ref{rootinj_lem}).

\medskip

Following \cite[Exp.~XXI, \S1.1]{SGA} we define a {\em root datum} to be an embedding
$\RS\hookrightarrow \cl^\vee$, $\al\mapsto \al^\vee$, 
of a non-empty finite subset $\RS$ of a
lattice $\cl$ into its dual $\cl^\vee$
such that 
\begin{itemize}
\item
$\RS \cap 2\RS=\emptyset$, 
$\al^\vee(\al)=2$ for all $\al \in \RS$, and
\item
$\beta - \al^\vee(\beta)\al \in \RS$ and 
$\beta^\vee - \beta^\vee(\al)\al^\vee \in \RS^\vee$ 
for all $\al,\beta \in \RS$, where $\RS^\vee$ denotes the image of
$\RS$ in $\cl^\vee$.
\end{itemize}
The elements of $\RS$ (resp. $\RS^\vee$) are called roots (resp. coroots). 
The sublattice of $\cl$ generated by $\RS$ 
is called the {\em root lattice} and is denoted by $\rl$. The rank of
$\cl_\Q$ is called the rank of the root datum.
A root datum is called irreducible if it is not a direct sum of root data of smaller ranks.
The sublattice of $\cl_\Q=\cl\ot_\Z \Q$
generated by all $\omega\in \cl_\Q$ such that 
$\al^\vee(\omega)\in \Z$ for all $\al\in\RS$ 
is called the {\em weight lattice} and is denoted by $\wl$. 
Observe that by definition $\rl\subseteq \wl$ and
$\rl\ot_\Z\Q=\wl\ot_Z \Q$.
A root datum is called semisimple if $\cl_\Q=\rl
\ot_\Z\Q$. 
From now on by a root datum we will always mean a semisimple one.
Observe that in this case $\rl\subseteq \cl\subseteq \wl$.

The root lattice $\rl$ admits a basis
$\Pi=\{\al_1,\ldots,\al_n\}$ 
such that 
each $\al \in \RS$ is a linear combination of $\al_i$'s with 
either all positive or all negative coefficients. 
So the set $\RS$ splits into two disjoint subsets 
$\RS = \RS_+ \amalg \RS_-$, where
$\RS_+$ (resp. $\RS_-$) is called the set of positive (resp. negative)
roots.  The roots $\al_i$ are called {\em simple roots}.

Given the set $\Pi$ we define the set of {\em fundamental weights}
$\{\omega_1,\ldots,\omega_n\}\subset \wl$
as $\al_i^\vee(\omega_j)=\de_{ij}$, 
where $\de_{ij}$ is the Kronecker symbol. 
Fundamental weights form a basis of the weight lattice $\wl$.
The matrix expressing simple roots in terms of fundamental weights 
is called the Cartan matrix of the root datum.

If $\cl=\wl$ (resp. $\cl=\rl$), 
then the root datum is called simply connected (resp. adjoint) and 
is denoted by $\mathcal{D}^{\sco}_n$ (resp. $\mathcal{D}_n^{\adj}$), 
where $\mathcal{D}=A,B,C,D,E,F,G$ is one of the Dynkin types and $n$
is its rank. Observe that an irreducible root datum is uniquely
determined by its Dynkin type and the lattice $\cl$.

Determinant of the Cartan matrix of an irreducible root
datum coincides with $|\wl/\rl|$, and can be found in the tables at
the end of \cite{Bo68} under the name ``indice de connexion''. For
future reference, we provide 
the list of determinants and the list
of simply connected torsion primes, i.e. prime divisors of the torsion index of the associated simply connected root datum; these can be found in \cite[Prop.~8]{Dem73}.

\smallskip

{\small
\begin{center}\label{detprimes_table}
\begin{tabular}{l||c|c|c|c|c|c|c|c|c}
Type (s. connected) & $A_l$ & $B_l$, $l\geq 3$ & $C_l$ & $D_l$, $l\geq 4$ & $G_2$ & $F_4$ & $E_6$ & $E_7$ & $E_8$ \\
Determinant & $l+1$ & $2$ & $2$ & $4$ & $1$ & $1$ & $3$ & $2$ & $1$ \\
Torsion primes & $\emptyset$ & $2$ & $\emptyset$ & $2$ & $2$ & $2,3$ & $2,3$ & $2,3$ & $2,3,5$ \\
\end{tabular}
\end{center}
}

\medskip

In several of the following statements, the root data $C_n^{\sco}$, $n\geq 1$ 
need to be treated separately. 
Note that $A_1^{\sco}=C_1^{\sco}$ and $B_2^{\sco}=C_2^{\sco}$ 
are among these special cases. 

\begin{lem} \label{completebasis_lem}
Any root $\al \in \RS$ can be completed to a basis of $\cl$ except if it is
a long root in an irreducible component $C_n^{\sco}$, $n\geq 1$, of the
root datum, in which case, it can still be completed to a
basis of $\cl$ over $\Z[\tfrac{1}{2}]$.
\end{lem}

\begin{proof}
We may assume that the root datum is irreducible. 

\smallskip

\noindent {\bf 1.} 
First, we can choose simple roots of $\rl$ as
$\Pi=\{\al=\al_1,\al_2,\ldots,\al_n\}$ 
by \cite[Ch.~IV, \S 1, Prop.~15]{Bo68}, 
hence, proving the statement in the adjoint case ($\cl=\rl$). 

\smallskip

\noindent {\bf 2.} 
If $\cl=\wl$, we express the root $\al$ in terms
of the fundamental weights corresponding to $\Pi$ as
$\al=\sum_{i=1}^n k_i \omega_i$. Looking at the Cartan matrix
we see that $k_i=-1$ for some $i$ except in type $G_2$ or 
if $\al=\al_n$ is the long root in type $C_n$, $n\geq 1$. 
In other words, except in these two cases 
$\al$ can be completed to a basis of $\wl$
by~Lemma~\ref{unimodular_lem} as $(k_1,...,k_n)$ is unimodular. 
Since the $G_2$ case is adjoint, it has already been considered. 
Finally, in the $C_n$ case, since $k_i=2$ for some $i$, the root $\al$ 
can be completed to a basis of $\wl$ after inverting $2$.

\smallskip

\noindent {\bf 3.} 
The remaining case $\rl\subsetneq \cl \subsetneq \wl$ 
can only happen in type $D_n$ or $A_n$, both with $n\geq 3$, 
see the above table. 

Pick a basis $\{\la_1,\ldots,\la_n\}$ of $\cl$ and 
express $\al=\sum_{i=1}^n k_i \la_i$. 
Then, complete $\al$ to a basis $\{\al=e_1,e_2,\ldots,e_n\}$ of $\wl$, 
as in the previous step. 
Let $A=(a_{ij})$ be the matrix whose columns express the basis $\{\la_i\}$
in terms of the basis $\{e_i\}$. 
Applying $A$ to the vector $(k_i)$ gives $\al=e_1$. 
So $\sum_{j=1}^n a_{j1}k_j = 1$ 
which means that the row $(k_1,\ldots,k_n)$ is unimodular 
and, therefore, $\al$ can be completed to a basis of $\cl$ 
by~Lemma~\ref{unimodular_lem}. 
\end{proof}

Consider a formal group algebra $\RcF$ associated to the lattice $\cl$
of the root datum
and a formal group law $F$ over $R$.

\begin{lem} \label{rootinj_lem}
For any root $\al\in \RS$ the element $x_{\al}$ is regular in
$\RcF$, except maybe if $\al$ is a
long root in an irreducible component $C_n^{\sco}$, $n\geq 1$, of the root datum. In that case, $x_\alpha$ is regular if and only if the formal series $2\cdot_F x$ is regular in $R\lbr x\rbr$, which always holds if $2$ is regular. 
\end{lem}

\begin{proof}If $\al$ is not a long root in $C_n^{\sco}$, $n\ge 1$, 
then by Lemma~\ref{completebasis_lem}, 
it can be completed to a basis of $\cl$, and 
$\RcF\cong R\lbr x_1,...,x_n\rbr$ with $x_1=x_\al$, so $x_\al$ is
regular  
by~Lemma~\ref{regulardegree1_cor}. 
On the other hand, if $\al$ is a long root in $C_n^{\sco}$, then it is part of a basis of roots, with associated fundamental weights part of a basis of $\cl$. 
If $n=1$, then $\al=2\omega$ where $\omega$ is the weight corresponding to $\alpha$. If $n\geq 2$, then $\al=2(\omega-\omega')$ where $\omega'$ is another weight. In both cases, we can therefore find an isomorphism $\RcF \simeq R\lbr x_1,\ldots,x_n\rbr$ such that $x_\alpha$ is mapped to $2\cdot_F x_1$, which is regular if and only if it is already regular in $R\lbr x_1\rbr$. If $2$ is regular, then $2\cdot_F x_1=2x_1 +x_1^2 y$ is regular by part~\eqref{onereg_item} of Lemma~\ref{regulardegree1_cor}.
\end{proof}

\begin{rem} For some formal group laws $x_\al$ can indeed be a regular element
even if $2$ is a zero divisor and $\al$ is a long root in $C_n^{\sco}$.
Take the multiplicative formal group law $F(u,v)=u+v-uv$ over
$R=\Z/2$ and the root datum $C_1^{\sco}$.
Then $x_\al=x_{2\omega}=2x_\omega +x_{\omega}^2=x_{\omega}^2$
is regular. 
\end{rem}

\begin{lem}
Let $\cl \subset \cl'$ be two lattices between $\rl$ and $\wl$. 
If the determinant of the Cartan matrix is regular in $R$, 
then the natural map $\RcF \to \FGR{R}{\cl'}{F}$ is injective. 
\end{lem}
 
\begin{proof}
Choose bases of $\cl$ and $\cl'$, and let $A$ be the matrix 
expressing the first basis in terms of the second. 
Then the determinant of $A$ divides the determinant of the Cartan matrix 
by composition of inclusions $\rl \subseteq \cl\subseteq
\cl'\subseteq \wl$. 
The lemma then follows from Lemma~\ref{regularinject_lem}.
\end{proof}

\begin{rem} 
Here is a counter-example to the injectivity 
when the determinant of the Cartan matrix is not regular. 
For the type $A_2$, take $R=\Z/3$ and the additive formal group law. 
The map from $\FGR{R}{\rl}{F} \to \FGR{R}{\wl}{F}$ is not injective: 
if $\{\al_1,\al_2\}$ are simple roots, then 
$x_{\al_1+2\al_2} = x_{\al_1}+2 x_{\al_2} \neq 0$ 
is sent to $x_{3 \omega_2}=3x_{\omega_2}=0$.
\end{rem}

\section{Localized formal group algebras}\label{sec:localized}

Consider a (semisimple) root datum $\RS\hookrightarrow \cl^\vee$.
Various operators on the formal group algebra $\RcF$, such as formal Demazure operators, are
defined using formulas involving formal division by elements
$x_\al$, $\al\in \RS$, see for example
\cite[Def. 3.5]{CPZ} in the case $\cl=\wl$. In the present section we study properties of the
localization of $\RcF$ at such elements. A key result here is
Lemma~\ref{intersection_lem} which will be used in the proof of 
Proposition~\ref{DFgenmodule_prop}.

\medskip

Let $\Qc$ denote the localization of $\RcF$ at the multiplicative
subset generated by 
elements $x_\al$ for all $\al\in \RS$ (if the multiplicative subset
contains zero, then $\Qc$ is trivial).

\begin{rem} Observe that by \cite[Def.~3.12]{CPZ} 
$x_{-\al}=x_\al(-1+x_{-\al}\kp_\al)$ for some
$\kp_\al \in \RcF$. Therefore, we
can replace $\RS$ by $\RS_+$ or $\RS_-$, since
$-1+x_{-\al}\kp_\al$ is invertible in $\RcF$. 
\end{rem}

\begin{lem} \label{localizationinj_lem}
If $2$ is regular in $R$ or the root datum doesn't have an irreducible
component $C_n^{sc}$, $n\ge 1$, then 
the localization map  $\RcF \to \Qc$ is injective.
\end{lem}

\begin{proof}
By Lemma~\ref{rootinj_lem}, the element $x_\al$ is regular for any
root $\al$, so we are localizing at a set consisting of regular elements.
\end{proof}
 
Let $r$ be a regular element of $R$. Then under the assumption of Lemma~\ref{localizationinj_lem}, there is a commutative diagram of inclusions
\begin{equation} \label{cartesian_diag}
\begin{split}
\xymatrix{
\RcF \ar@{}[d]|{\rotatebox{270}{$\textstyle \subseteq$}} \ar@{}[r]|-{\textstyle \subseteq} & \RcF\lr \ar@{}[d]|{\rotatebox{270}{$\textstyle \subseteq$}} \ar@{}[r]|-{\textstyle \subseteq} & \FGR{R\lr}{\cl}{F} \ar@{}[d]|{\rotatebox{270}{$\textstyle \subseteq$}} \\
\Qc \ar@{}[r]|-{\textstyle \subseteq} & \Qc\lr \ar@{}[r]|-{\textstyle \subseteq} & Q_{R\lr,\cl}
}
\end{split}
\end{equation}
This can be seen by choosing a basis of $\cl$ and, hence,
identifying the respective formal group algebra with a ring of power series.
Observe that if $r$ is non-invertible, an element $\sum_{i\ge 0} \frac{x^i}{r^i}$ while being in
$R\lr\lbr x \rbr$ is not in $R\lbr x \rbr
\lr$, so the inclusions $\RcF \subsetneq
\RcF\lr \subsetneq \FGR{R\lr}{\cl}{F}$ are proper.

\begin{lem} \label{goingup_lem}
Let $r$ be a regular element of $R$
and let $\al$ be a root.
Assume that $2$ is invertible in $R$ or
$\al$ is not a long root of an irreducible component
$C_n^{\sco}$, $n\geq 1$.
If $u \in \FGR{R\lr}{\cl}{F}$ and $x_\al u \in \RcF\lr$, 
then $u \in \RcF\lr$. 
\end{lem}

\begin{proof}
By Lemma~\ref{completebasis_lem} we complete $\al$ to a basis of
$\cl$. 
Observe that if $\al$ is a long root of $C_n^{\sco}$, then 
$x_\al\in 2x_\omega+\IF^2$. Therefore, the isomorphism 
$\RcF\simeq R\lbr x_1,\ldots,x_n\rbr$ determined by the choice of a basis 
can be modified so that it sends $x_\al$ to $x_1$.

The lemma then follows from the fact that 
if $u \in R\lr\lbr x_1,\ldots,x_n\rbr$ and $x_1 u \in R \lbr x_1,\ldots,x_n\rbr\lr$, then $u \in R\lbr x_1,\ldots, x_n\rbr\lr$.
\end{proof}

\begin{cor}\label{localintersection_cor} 
Let $r$ be a regular element of $R$.
Assume that $2$ is invertible in $R$ or the root datum
doesn't have an irreducible component $C_n^{\sco}$, $n\ge 1$.
Then we have
$(\FGR{R\lr}{\cl}{F})\cap \Qc\lr =
  \RcF\lr$  in $Q_{R\lr,\cl}$, i.e. the right
  square of \eqref{cartesian_diag} is cartesian.
\end{cor}

\begin{lem}\label{intersection_lem}
Let $r$ be a regular element of $R$.
Assume that $2$ is invertible in $R$ or the root datum
doesn't have an irreducible component $C_n^{\sco}$, $n\ge 1$.
Then we have
 $(\RcF\lr)\cap \Qc = \RcF$ in $\Qc\lr$, i.e. the left
 square of \eqref{cartesian_diag} is cartesian. 
\end{lem}

\begin{proof} 
Since the localization map is injective by
Lemma~\ref{localizationinj_lem}, it is enough to show that if $u,v \in \RcF$ satisfy
$\frac{v}{\prod_\al x_\al} = \frac{u}{r^n}$, $\al\in\RS$, or equivalently, $r^n v =
u\prod_\al x_\al$ in $\RcF$, then $r^n$ divides $u$. 
In $\FGR{(R/r^n)}{\cl}{F}=\RcF/r^n$, we have $r^n v = 0$, and since by Lemma~\ref{rootinj_lem} the element $x_\al$ is regular in $\FGR{(R/r^n)}{\cl}{F}$ for any $\alpha \in \RS$, we also have $u=0$. Thus $u$ is divisible by $r^n$.
\end{proof}

\section{Formal Demazure operators} \label{FormalDemazureOperators_sec}

We now introduce formal Demazure
operators following the approach 
mentioned in \cite[Rem.~3.6]{CPZ}. The formula \eqref{productform_eq}
will be extensively used in section~\ref{sec:copraffine} in various computations involving coproducts.

\medskip

Consider a root datum $\RS\hookrightarrow \cl^\vee$. For each $\al \in \RS$
we define a $\Z$-linear automorphism of $\cl$ called a {\em reflection} by
\[
s_\al\colon \la \mapsto \la -
\al^\vee(\la)\al,\quad \la\in \cl.
\]
The subgroup of linear automorphisms of $\cl$ generated 
by reflections $s_\al$, $\al\in \RS$,
is called the {\em Weyl group} of the root datum and is denoted by $W$.
Observe that $W$ depends only on the
Dynkin type.

Fix a formal group law $F$ over $R$.  There is an induced action of
the Weyl group $W$ on $\RcF$ via $R$-algebra automorphisms defined by
$w(x_\la) = x_{w(\la)}$ for all $w \in W$ and $\la \in
\cl$.

\begin{lem} If $x_\al$ is regular in $\RcF$, then the difference 
$u-s_\al(u)$ is uniquely divisible by $x_\al$ for each $u\in
\RcF$
and $\al\in \RS$.
\end{lem}

\begin{proof} Divisibility follows from the proof of \cite[Cor.~3.4]{CPZ}. Uniqueness follows from the regularity of $x_\alpha$.
\end{proof}

\begin{defi}\label{formaldemazureop_defi}
For each $\al\in\RS$ such that $x_\al$ is regular in $\RcF$
we define an $R$-linear operator 
$\De_\al$ on $\RcF$, called a {\em formal Demazure operator},
\[
    \De_\al(u ):=\tfrac{u -s_\al(u) }{x_\al},\quad u \in \RcF.
\]
\end{defi}

\begin{rem} Observe that in \cite[Def.~3.5]{CPZ} we defined the
  Demazure
operator on $\RcF$ for $\cl=\wl$ by specializing the Demazure operator
defined over the Lazard ring. If $x_\al$ is regular, then
this definition coincides with Definition~\ref{formaldemazureop_defi}.
\end{rem}

\begin{defi}\label{asmreg_defi}
We say that $\RcF$ is $\RS$-\emph{regular} 
if for each $\al\in \RS$, the element $x_\al$ is regular in $\RcF$.
\end{defi}

Unless otherwise stated, we shall always assume that the formal group
algebra $\RcF$ is $\RS$-regular. Observe that this immediately implies
that $\RcF$ injects into the localization $\Qc$ (cf.~Lemma~\ref{localizationinj_lem}).

\begin{rem}\label{rem2reg_rem}
Note that Lemma~\ref{rootinj_lem} implies that, $\RcF$ is $\RS$-regular if $2$ is regular in $R$ or if the root datum doesn't contain an irreducible component $C_n^{\sco}$.
\end{rem}

We fix a set of simple roots $\{\al_1,\ldots,\al_n\}$. 
Let $s_i=s_{\al_i}$, $i\in [n]=\{1,\ldots,n\}$ denote the corresponding (simple)
reflections. One of the basic facts concerning the Weyl group $W$ is
that it is generated by simple reflections.
Given a sequence $I=(i_1,...,i_l)$ of length $|I|=l$ with $i_j\in [n]$, we
set $w(I)=s_{i_1}\ldots s_{i_l}$ to be the product of the respective
simple reflections. 
A sequence $I$ is called a {\em reduced sequence} of $w\in W$ if $I$ is of minimal length among all sequences $J$ such that $w=w(J)$. 
The length of $I_w$
is called the length of $w$ and is denoted by $\ell(w)$.
Given a sequence $I$ we define
\[
  \De_{I}=\De_{i_1}\circ...\circ \De_{i_l},
  \quad\text{ where }\De_i:=\De_{\al_i}.
\]
 
\begin{rem}
If $F$ is the additive or multiplicative formal group law,
then $\De_{I_w}$ doesn't depend 
on a choice of the reduced sequence of $w$. 
In this case, it coincides with the classical Demazure operator $\De_w$ of
\cite[\S3 and~\S9]{Dem73}.
For other formal group laws
$\De_{I_w}$ depends on a choice of $I_w$ 
(see \cite[Thm.~3.9]{CPZ}).
\end{rem}

\begin{defi}
We define $R$-linear operators $B_i^{(j)}\colon\RcF\to \RcF$,
where $j\in\{-1, 0,1\}$ and $i\in[n]$, as
\[
  B_i^{(-1)}:=\De_i, \quad B_i^{(0)}:=s_i, \quad\text{and}\quad
  B_i^{(1)}:=\text{multiplication by }(-x_i):=-x_{\al_i}.
\]
\end{defi}
Observe that we have $B^{(j)}_i(\IF^m)\subset \IF^{m+j}$, where $\IF$
is the augmentation ideal. 

\medskip

Let $I=(i_1,\ldots,i_l)$ be a sequence of length $l$, and let $E$ be a
subset of $[l]$. We denote by $I_{|E}$ the subsequence of $I$
consisting of all $i_j$'s with $j\in E$.

\begin{lem}\label{productform_lem} Let $I=(i_1,\ldots, i_l)$. Then for any $u,v\in \RcF$ we have
\begin{equation} \label{productform_eq}
\De_I(uv)=\sum_{E_1,E_2\subset [l]} p_{E_1,E_2}^I
\De_{I_{|E_1}}(u)\De_{I_{|E_2}}(v)
\end{equation}
where $p_{E_1,E_2}^I=B_1\circ...\circ B_l (1)\in \IF^{|E_1|+|E_2|-l}$
with the operator $B_j\colon \RcF\to \RcF$ defined as 
\[
B_j=
\begin{cases}
B_{i_j}^{(1)}\circ B_{i_j}^{(0)}, & \text{if $j\in E_1\cap E_2$,} \\
B_{i_j}^{(-1)}, & \text{if $j\notin E_1\cup E_2$,} \\
B_{i_j}^{(0)},  & \text{otherwise.}
\end{cases}
\]
\end{lem}

\begin{proof} We proceed by induction on the length $l$ of $I$.
For $l=1$ and $I=(i)$ the formula holds since 
\[
\begin{split}
\De_i(uv) & =\De_i(u)v+u\De_i(v)-x_i\De_i(u)\De_i(v) \\
 & =-\De_i(1)uv +s_i(1)u\De_i(v)+s_i(1)\De_i(u)v - x_i \De_i(u)\De_i(v)
\end{split}
\]
by \cite[Prop.~3.8.(4)]{CPZ} and because $\De_i(1)=0$ and $s_i(1)=1$.

Setting $I'=(i_2,\ldots,i_l)$, we obtain by induction
\[
\De_I(uv)=\De_{i_1}(\De_{I'}(uv))=\De_{i_1}\Big(\sum_{E_1,E_2\subset [l-1]}
p_{E_1,E_2}^{I'} \De_{I'_{|E_1}}(u) \De_{I'_{|E_2}}(v)\Big).
\]
For fixed $E_1,E_2\subset [l-1]$, by \cite[Prop.~3.8.(4)]{CPZ} we have
\[
\begin{split}
\De_{i_1}\big(p_{E_1,E_2}^{I'}\De_{I'_{|E_1}}(u)\De_{I'_{|E_2}}(v)\big)= &
\De_{i_1}(p_{E_1,E_2}^{I'})\De_{I'_{|E_1}}(u)\De_{I'_{|E_2}}(v)  \\
 & +s_{i_1}(p_{E_1,E_2}^{I'})\De_{i_1}\big(\De_{I'_{|E_1}}(u)\De_{I'_{|E_2}}(v)\big), 
\end{split}
\]
and
\[
\begin{split}
\De_{i_1}\big(\De_{I'_{|E_1}}(u)\De_{I'_{|E_2}}(v)\big) = & (\De_{i_1}\circ\De_{I'_{|E_1}})(u)\De_{I'_{|E_2}}(v)+
\De_{I'_{|E_1}}(u)(\De_{i_1}\circ\De_{I'_{|E_2}})(v)\\
& -x_{i_1} (\De_{i_1}\circ\De_{I'_{|E_1}})(u) (\De_{i_1}\circ\De_{I'_{E_2}})(v).
\end{split}
\]
Combining these two identities, we obtain 
\begin{multline*}
\De_{i_1}\big(p_{E_1,E_2}^{I'}\De_{I'_{|E_1}}(u)\De_{I'_{|E_2}}(v)\big)
= B_{i_1}^{(-1)}(p^{I'}_{E_1,E_2})\cdot \De_{I'_{|E_1}}(u)\De_{I'_{|E_2}}(v)  \\
+B_{i_1}^{(0)}(p^{I'}_{E_1,E_2})\cdot [(\De_{i_1}\circ\De_{I'_{|E_1}})(u)\De_{I'_{|E_2}}(v)+\De_{I'_{|E_1}}(u)(\De_{i_1}\circ\De_{I'_{|E_2}})(v)]
\\
+(B_{i_1}^{(1)}\circ B_{i_1}^{(0)})(p^{I'}_{E_1,E_2}) \cdot (\De_{i_1}\circ\De_{I'_{|E_1}})(u) (\De_{i_1}\circ\De_{I'_{|E_2}})(v)
\end{multline*}
and using the notation $E+1=\{e+1 \mid e\in E\}$ for a set of integers
$E$ we get
\begin{multline*}
= B_{i_1}^{(-1)}(p^{I'}_{E_1,E_2})\cdot \De_{I_{|E_1+1}}(u)\De_{I_{|E_2+1}}(v)  \\
+B_{i_1}^{(0)}(p^{I'}_{E_1,E_2})\cdot [\De_{I_{|\{1\}\cup (E_1+1)}})(u)\De_{I_{E_2+1}}(v)+\De_{I_{|E_1+1}}(u)\De_{I_{|\{1\}\cup (E_2+1)}}(v)]
\\
+(B_{i_1}^{(1)}\circ B_{i_1}^{(0)})(p^{I'}_{E_1,E_2}) \cdot (\De_{I_{|\{1\}\cup (E_1+1)}})(u) \De_{I_{|\{1\}\cup (E_2+1)}}(v). 
\end{multline*}
Since a subset of $[l]$ is of the form  $E+1$ or $\{1\}\cup (E+1)$ with $E\subseteq [l-1]$, the result follows by induction. 
\end{proof}

\begin{rem}\label{vanp_rem}
For any simple root $\al$ we have $s_\al(1)=1$ and $\De_\al(1)=0$. Therefore, if there is an operator of type $B^{(-1)}$ applied to $1$ before applying all the operators of type $B^{(1)}\circ B^{(0)}$, then $p^I_{E_1,E_2}=0$. In particular,  if $E_1\cap E_2=\emptyset$ and $E_1\cup E_2\neq [l]$, then $p^I_{E_1,E_2}=0$.
\end{rem}

\section{Twisted formal group algebra and Demazure elements}\label{sec:twisted}

We recall the notions of twisted formal group algebra and formal affine Demazure algebra introduced in \cite[\S6]{HMSZ}. The purpose of this section is to generalize
\cite[Prop.~4.3]{KK} and \cite[Prop.~2.6]{kk} to the context of arbitrary formal group laws
(see Lemma~\ref{Deldecomp_lem}), hence, providing a general formula expressing a product of Demazure
elements in terms of the canonical basis of the twisted formal group algebra.

\medskip

As in section~\ref{sec:localized},
given a root datum $\RS \hookrightarrow \cl^\vee$
consider the localization $\qQ=\Qc$ of the formal group algebra $\sS=\RcF$. 
Since the Weyl group $W$
preserves the set $\RS$, its action on $\sS$ extends to $\qQ$.
Following \cite[\S4.1]{KK} and \cite[Def.~6.1]{HMSZ}, 
we define the \emph{twisted formal group algebra} 
to be the $R$-module $\QW := \qQ \ot_R R[W]$ with the
multiplication given by
\[
    (q\de_w) (q'\de_{w'}) = qw(q')\de_{ww'}\; 
\text{ for all } w,w'\in W\text{ and } q,q' \in \qQ
\]
  (extended by linearity), where $\de_w$ denotes the element in $R[W]$
  corresponding to $w$ 
(so we have $\de_w \de_{w'} = \de_{ww'}$ for $w,w' \in W$) and
$\de_1$ denotes $1$. 
Observe that $\QW$ is a free left $\qQ$-module (via the left
multiplication) 
with basis $\{\de_w\}_{w\in W}$,
but $\QW$ is not a $\qQ$-algebra
as $\de_1 \qQ=\qQ\de_1$ is not central in $\QW$.

\begin{rem} The twisted formal group algebra can be defined for any
 $W$-action on $\qQ$, where $W$ is a unital monoid and
$w(q_1\cdot q_2)=w(q_1)\cdot w(q_2)$ for any $q_1,q_2\in Q$ and
  $w\in W$.
\end{rem}

\begin{rem} In \cite[\S4.1]{KK} (resp. \cite[\S2.1]{kk}),
$\qQ$ denoted the field of fractions of a symmetric algebra (resp. of an integral group ring) of the weight lattice $\wl$ and $\QW$ was defined using the right $Q$-module notation, i.e. $\QW=\Z[W]\ot_\Z \qQ$.
Using our terminology, the case of \cite{KK} (resp. of \cite{kk}) corresponds to
the additive formal group law (resp. multiplicative periodic formal group law) and the simply connected root datum. 

In \cite{HMSZ}, $\qQ$ denoted the localization of $R\lbr \wl \rbr_F$ at all elements $x_\la$, where $\la \in \wl$.
\end{rem}
 
\begin{defi}
Following \cite[$I_{24}$]{KK} and \cite[Def.~6.2]{HMSZ}, for each 
root $\al \in \RS$ we define the \emph{Demazure element}
\[
X_\al:=x_\al^{-1}(1-\de_{s_\al})=  
x_\al^{-1}-\de_{s_\al} x_{-\al}^{-1} \in \QW.
\]
Given a set of simple
roots $\{\al_1,\ldots,\al_n\}$ for any sequence $I=(i_1,...,i_l)$ from $[n]$ we denote
\[
\de_I=\de_{s_{i_1}...s_{i_l}}\; \text{ and }\; \Del_I:=\Del_{i_1}\cdot
\ldots\cdot \Del_{i_l},\text{ where }\Del_i=\Del_{\al_i}.
\]
There is an {\em anti-involution} $(\text{-})^t$ on the $R$-algebra $\QW$ defined by
\[
 q\de_w \mapsto (q\de_w)^t:=w^{-1}(q)\de_{w^{-1}},\quad w\in W,\; q \in \qQ.
\]
Observe that $(qx)^t=x^tq$ for $x\in \QW$ and $q\in \qQ$, so it is
neither 
right $\qQ$-linear nor left $\qQ$-linear.
\end{defi}

We obtain the following generalization of \cite[Prop. 4.3]{KK}:

\begin{lem}\label{Deldecomp_lem}
Given a reduced sequence $I_v$ of $v\in W$ of length $l$ let 
\[
\Del_{I_v}=\sum_{w\in W} a_{v,w} \de_{w}=\sum_{w \in W}
\de_w a'_{v,w}\] for some $a_{v,w}, a'_{v,w} \in \qQ$. Then
\begin{enumerate}
\item \label{zerounless_item} $a_{v,w}=0$ unless $w \leq v$ with
  respect to the Bruhat order on $W$,
\item \label{extremecoeffs_item} $a_{v,v}=  (-1)^l \prod_{\al \in v
    (\RS_-) \cap \RS_+} x_{\al}^{-1}=a'_{v,v^{-1}}$,
\item \label{reverse_item} $a'_{v,w}=0$ unless $w \geq v^{-1}$.
\end{enumerate}  
\end{lem}

\begin{proof} We proceed by induction on the length $l$ of $v$.

The lemma obviously  holds for $I_v =\emptyset$ the empty sequence,
since $X_\emptyset =1$. 

Let $I_v=(i_1,\ldots i_l)$ be a reduced sequence of $v$ and let $\beta=\al_{i_1}$. Then
$(i_2,\ldots,i_n)$ is a reduced sequence of $v'=s_\beta v$ and we have
\begin{enumerate}[\indent(1)]
\item \label{stableBruhat_item} $w \leq v'$ implies $w \leq v$ and $s_\beta w \leq v$;
\item \label{negroots_item} $\{\beta\} \cup s_\beta\big( v'(\RS_-)\cap \RS_+\big) = v(\RS_-) \cap \RS_+$;
\item \label{inverseBruhat_item} $w^{-1} \leq v$ if and only $w \geq v^{-1}$.
\end{enumerate}
Indeed, the properties~\eqref{stableBruhat_item} and \eqref{inverseBruhat_item}
are consequences of the fact that elements smaller than $v$ are the
elements $w$ obtained by taking a subsequence of a reduced sequence of
$v$, by \cite[Th. 1.1, III, (ii)]{Deo77}. Property
\eqref{negroots_item} follows from \cite[Ch. VI, \S 1, No 6,
Cor. 2]{Bo68}. 

We then compute
\begin{equation*}
\begin{split}
\Del_{I_v} & =x_\beta^{-1} (1-\de_\beta) \sum_{w\leq v'} a_{v',w}\de_w = \sum_{w \leq v'} x_\beta^{-1}a_{v',w} \de_w -\sum_{w\leq v'} x_\beta^{-1} s_\beta(a_{v',w})\de_{s_\beta w}
\\ & \equalbyeq{stableBruhat_item} - x_\beta^{-1} s_\beta(a_{v',v'})\de_v +\sum_{w < v} a_{v,w} \de_w. 
\end{split}
\end{equation*}
So part~\eqref{zerounless_item} holds and 
$a_{v,v} = -x_\beta^{-1}s_\beta(a_{v',v'})$. 
Hence, the expression of $a_{v,v}$ in part~\eqref{extremecoeffs_item}
follows from property~\eqref{negroots_item}. 
Property \eqref{inverseBruhat_item} implies \eqref{reverse_item} by applying the anti-involution sending $\delta_w$ to $\delta_{w^{-1}}$ and, thus, identifying $a'_{v,w}=a_{v,w^{-1}}$.
\end{proof}

\begin{rem}
Observe that the coefficient $a_{v,v}$ doesn't depend on the choice of
a reduced sequence $I_v$ of $v$.
\end{rem}

\begin{cor}(cf. \cite[Cor.~4.5]{KK}) \label{Delbasis_cor} For each $w\in W$, let $I_w$ be a reduced sequence of $w$. The elements $(\Del_{I_w})_{w\in W}$ form a basis of $\QW$ as a left (resp. right) $\qQ$-module.
The element $\de_v$ decomposes as $\sum_{w\leq v} b_{w,v}
\Del_{I_w}$ with $b_{w,v}$ in $\qQ$. 
Furthermore, \[
b_{v,v}=(-1)^l \prod_{\al \in v (\RS_-) \cap \RS_+} x_{\al}.
\]
\end{cor} 

\begin{proof}
The matrix $(a_{v,w})_{(v,w)\in W^2}$, resp. $(a'_{v,w})_{(v,w)\in
  W^2}$, is triangular with invertible coefficients on the diagonal
(resp. the anti-diagonal) by Lemma~\ref{Deldecomp_lem}. It expresses
elements $\Del_{I_w}$ in terms of the basis $(\de_w)_{w\in W}$ of the left (resp. right) $\qQ$-module $\QW$.

The decomposition of $\de_w$ follows from the fact that the inverse
of a triangular matrix is also triangular with inverse coefficients on the diagonal.
\end{proof}

We are now ready to define a key object of the present paper.

\begin{defi}\label{affdemalg_defi}
We define the {\em formal affine
Demazure algebra} $\DcF$ to be the $R$-subalgebra of $\QW$ generated by
elements of the formal group algebra $\sS$ and by Demazure elements
$\Del_\al$ for all $\al\in \RS$.
\end{defi}

The following lemma shows that it is the same object as the one considered in \cite[Definition~6.3]{HMSZ}.
\begin{lem}\label{simplerlem_lem}
The Demazure algebra $\DcF$ coincides with the $R$-subalgebra of $\QW$
generated by elements of $\sS$ and Demazure elements $\Del_i$, for $i\in [l]$.
\end{lem}
\begin{proof} Since the Weyl group is generated by reflections and
$\de_{s_i} = 1-x_{\al_i} \Del_i$, 
we have $R[W]\subseteq \DcF$. 
Since any $\al\in \RS$ can be written as $\al=w(\al_i)$ for some
simple root $\al_i$,
the lemma then follows from the formula $\de_w\Del_{\al_i} \de_{w^{-1}}=\Del_{w(\al_i)}$.
\end{proof}

Let us now mention the functorial behaviour of these constructions. 
\begin{prop} \label{functorial_prop} The algebra $S$ and its localization $Q$ as well as the algebra $\QW$ and its subalebra $\DcF$ are functorial in
\begin{itemize}
\item morphisms of root data i.e.\ morphisms of lattices $\phi:\Lambda \to \Lambda'$ sending roots to roots and such that $\phi^\vee(\phi(\alpha)^\vee)=\alpha^\vee$;
\item morphisms of rings $R \to R'$ sending the formal group law $F$ over $R$ to the formal group law $F'$ over $R'$.
\end{itemize}
\end{prop}
\begin{proof}
The formal group algebra $S$ is functorial along both of these situations by \cite[Lemma 2.6]{CPZ}, and since variables $x_\alpha$ are preserved, the functoriality extends to the localization $Q$ of $S$. A morphism of root data induces a morphism of Weyl groups, thus the functoriality extends to $\QW$. Finally, the elements $\Del_\alpha$ are preserved by this functoriality, and therefore so is the subalgebra $\DcF$.
\end{proof}

\section{Invariant presentation of the formal affine Demazure algebra}\label{sec:invar}

In the present section using the action of Demazure elements on $\RcF$
via
Demazure operators we consider the algebra $\tDcF$ of elements of $\QW$ fixing $\sS$. It is closely related to the formal
affine Demazure algebra $\DcF$. We show that the algebra $\tDcF$ is a
free $\RcF$-module with a basis given by products of Demazure elements
(see Corollary~\ref{basistDcF_cor}), hence, generalizing \cite[Thm.4.6.(a)]{KK}.

\medskip

The localization $\qQ=\Qc$ of $\sS=\RcF$ 
has a structure of left $\QW$-module defined
by (cf. \cite[$I_{33}$]{KK})
\[
(\de_w q)q' = w(qq')\text{ for }w\in W\text{ and }q,q'\in \qQ.
\]
Let $\tDcF$ denote the $R$-subalgebra of $\QW$ preserving $\sS$ when
acting on the left, i.e.
\[
\tDcF:=\{x\in \QW\mid x \cdot \sS\subseteq \sS\}.
\]
By definition we have $\sS\subset \tDcF$ and  $\Del_\al \in \tDcF$, since $\Del_\al$ acts on $\sS$ by the
Demazure operator $\De_\al$. Therefore, in this case $\DcF\subseteq \tDcF$.

\medskip

Let $\tor$ be the {\em torsion index} of the root datum 
as defined in \cite[\S 5]{Dem73}. Its prime factors are the torsion
primes listed in the table~in~\S\ref{detprimes_table} together with the prime divisors of $|\wl/\cl|$, by \cite[\S 5, Prop. 6]{Dem73}.

Let $I_0$ be a reduced sequence of the longest element $w_0$ and 
let $N=\ell(w_0)$. Recall that by \cite[5.2]{CPZ} 
there exists an element $u_0\in \IF^N\subset \sS$ such that 
$\aug \De_{I_0}(u_0)=\tor$. 
Moreover, $u_0$ satisfies the following property: if $|I|\le N$, then
\[
\aug \De_I(u_0)=\left\{\begin{array}{ll}\tor, & \text{if } I \text{ is
      a reduced sequence of }w_0,\\
0, & \text{otherwise.}\end{array}\right.
\]

Let us recall a key result for future reference.
\begin{lem} \label{Deltamatrixinvert_lem}
For each $w\in W$, let $I_w$ be a reduced sequence of $w$. Then if the torsion index $\tor$ is invertible in $R$, the matrix 
$\left(\De_{I_v}\De_{I_w}(u_0)\right)_{(v,w)\in W\times W}$ with
coefficients in $\sS$ is invertible. Thus, if $\tor$ is just regular, the kernel of this matrix is trivial.
\end{lem}
\begin{proof}
It readily follows from the above property of $u_0$, see \cite[Prop.~6.6]{CPZ} and its proof. 
\end{proof}

\begin{prop} \label{DFgenmodule_prop}
Assume that $\tor$ is regular in $R$. Also assume
that $2\in R^\times$ or that either all or none of the irreducible components of the root
datum are $C_n^{\sco}$, $n\geq 1$.

Then the $R$-algebra $\tDcF$ is the left $\sS$-submodule of
  $\QW$ 
generated by $\{\Del_{I_w}\}_{w\in W}$.
\end{prop}
\begin{proof}
Observe that the hypotheses of the proposition imply that $\sS$ is $\Sigma$-regular, so $\tDcF$ is well-defined.

By Corollary~\ref{Delbasis_cor} we can write $y\in \tDcF$ as
$y=\sum_v c_v \Del_{I_v}$, where $c_v\in\qQ$.
It is enough to prove that $c_v \in \sS$ for each $v$. 
Apply $y$ successively to the elements $\De_{I_w}(u_0)$, for all $w\in
W$. By definition of $\DcF$ all resulting elements $\sum_v c_v
\De_{I_v}\De_{I_w}(u_0)$ are in $\sS$. 
By Lemma \ref{Deltamatrixinvert_lem}, the matrix 
$\left(\De_{I_v}\De_{I_w}(u_0)\right)_{(v,w)\in W\times W}$ with
coefficients in $\sS$ becomes invertible
after inverting $\tor$. This implies that
$c_v \in \FGR{R[\tfrac{1}{\tor}]}{\cl}{F}$ for each $v\in W$. 

If all irreducible components of the root system are $C_n^{\sco}$,
$n\geq 1$, we are finished since $\tor=1$.
In the remaining two cases ($2$ is invertible or there are no
components $C_n^{\sco}$) we have $\FGR{R[\tfrac{1}{\tor}]}{\cl}{F} \cap \qQ = \sS$ by Corollary~\ref{localintersection_cor} and Lemma~\ref{intersection_lem}.
\end{proof}
 
\begin{cor}\label{basistDcF_cor}
Under the hypotheses of Proposition~\ref{DFgenmodule_prop} the elements $\{X_{I_w}\}_{w\in W}$ form a basis of
the left $\sS$-module $\tDcF$.
\end{cor}

\begin{proof}
The elements $\Del_{I_w}$ are independent over $\sS$, since they are
independent over $\qQ$ and the map $\sS \to \qQ$ is injective by Lemma~\ref{localizationinj_lem}.
\end{proof}

\begin{example}
Here is an example where the torsion index $\tor=2$ is regular in $R=\Z$ while 
the remaining assumptions of Proposition~\ref{DFgenmodule_prop} are not satisfied
and its conclusion does not hold. 

Consider the additive formal group law over $\Z$ and a direct sum of root data
$C_1^{\adj}$ and $C_1^{\sco}$ 
with the respective lattices $\rl$ and $\wl'$. Let $\al$ be the
simple root of $\rl$, and let $\omega'$ be the fundamental weight of
$\wl'$. Consider the isomorphism $\FGR{\Z}{\cl\oplus \cl'}{F}\simeq
\Z\lbr x,y \rbr$ defined by $x_\al\mapsto x$ and
$x_{\omega'}\mapsto y$. Observe
that 
$\qQ\simeq \Z\lbr x,y \rbr [\tfrac{1}{x},\tfrac{1}{2y}]$, so it contains
$\tfrac{1}{2}$. 
 
Consider the simple reflection $s_\al$. By definition 
$s_\al(y)=y$ and $s_\al(x)=-x$. 
Therefore, $\De_\al$ acts on $\qQ$ linearly over $\Z\lbr y\rbr$ and 
$\De_\al(x^i)=(x^i +(-x)^i)/x$ for each $i\ge 1$. 
So if $z=\sum_{i\ge 0} a_i x^i$ with $a_i \in \Z\lbr y\rbr$, then
$\De_\al(z)=\sum_{i\ge 1} a_i (1+(-1)^i) x^{i-1}$ is divisible by
$2$. Hence,  $\tfrac{1}{2}\Del_\al$ is in $\tDcF$, while it is not
in the left $\sS$-submodule, since $\tfrac{1}{2} \notin
\Z\lbr x,y\rbr$. Observe that the proposition holds if $\Z$ is replaced by
$\Z[\tfrac{1}{2}]$.
\end{example}

\section{Presentation in terms of generators and relations}\label{sec:present}

In the present section we describe the formal affine Demazure algebra
$\DcF$
in terms of generators and relations. The main result 
(Theorem~\ref{presentationDF_theo})
generalizes \cite[Thm.~6.14]{HMSZ}. We also establish an
isomorphism (Theorem~\ref{DF=ED_theo}) 
between the formal affine Demazure algebra and the algebra
of Demazure operators on $\RcF$ of \cite[\S4]{CPZ}.

\begin{lem} \label{decompdifference_lem} 
Let $I$ and $I'$ be reduced sequences of the same element $w\in W$. Then in the formal affine Demazure algebra $\DcF$ we have
\[
\Del_{I} -\Del_{I'}= \sum_{v<w} c_v \Del_{I_v}\;\text{ for some }c_v \in S.
\]
\end{lem}
 
\begin{proof}
The twisted formal group algebra is functorial in
the root datum and in morphisms of rings preserving formal group laws by Proposition \ref{functorial_prop}, and the functoriality preserves the elements $\Del_\alpha$. We may therefore assume that each irreducible component of
our root datum is adjoint (hence not $C_n^{\sco}$) and that the base ring is the Lazard ring, in which all integers are regular. In particular, all the hypotheses~of~\ref{DFgenmodule_prop} are satisfied.
Then, by Proposition~\ref{DFgenmodule_prop}, 
the difference $\Del_{I}-\Del_{I'}\in \tDcF$ can be written as a linear
combination $\sum_{w \in
  W} c_v \Del_{I_v}$, $c_v\in \sS$.
Since $\sS$ injects into $\qQ$ by Lemma~\ref{localizationinj_lem}, it
suffices to check that $c_v = 0$ unless $v\leq w$ (in the Bruhat
order) when this linear combination is the one obtained in $\QW$ by Corollary~\ref{Delbasis_cor}. We therefore compute:
\begin{equation*}
\begin{aligned}
\Del_{I}-\Del_{I'} 	& \equalby{Lem. \ref{Deldecomp_lem} \eqref{zerounless_item}} \sum_{u\leq w} a_u \de_u -\sum_{u\leq w} a'_u \de_u 
 \equalby{Lem. \ref{Deldecomp_lem} \eqref{extremecoeffs_item}} \sum_{u<w} (a_u - a'_u) \de_u \\ 
 			& \equalby{Cor. \ref{Delbasis_cor}} \sum_{u<w} (a_u - a'_u) \sum_{v\leq u} b_v \Del_{I_v}  = \sum_{v<w} b_v\Big(\sum_{v\leq u <w} (a_u-a'_u)\Big) \Del_{I_v}. \qedhere
\end{aligned}
\end{equation*}
\end{proof}

Let us examine commutation relations between Demazure elements of
simple roots of root datum of rank $1$ or $2$. 

\begin{example}
For any root $\alpha$ and any $q \in Q$, we have \cite[Lemma~6.5]{HMSZ}
\begin{equation} \label{Qcommute_eq}
\Del_\alpha q = \De_\alpha(q) + s_\alpha(q) \Del_\alpha.
\end{equation}
\end{example}

\begin{example}
In the rank $1$ case we have \cite[(6.1)]{HMSZ}
\begin{equation} \label{rankone_eq}
\Del_\al^2= \kp_\al \Del_\al = \Del_\al \kp_\al
\end{equation}
where $\kp_\al=\frac{1}{x_\al}+\frac{1}{x_{-\al}}$ actually
lives in $\sS$ \cite[Def. 3.12]{CPZ}. 
\end{example}

\begin{example}\label{ranktwo_ex}
For the rank $2$ case let $\al_1$ and $\al_2$ be simple roots and let
$W$ be the Weyl group generated by $s_1=s_{\al_1}$ and
$s_2=s_{\al_2}$. 
Let $m=2,3,4$ or $6$ be the order of $s_1 s_2$, i.e. the length of the longest element 
$$w_0^{\al_1,\al_2}=\underbrace{s_1 s_2 s_1 \cdots}_{\text{$m$ times}} = \underbrace{s_2 s_1 s_2 \cdots}_{\text{$m$ times}}.$$
Any element $w \neq w_0^{\al_1,\al_2}$ in $W$ has a unique reduced sequence, of the
form either $(1,2,1 \cdots)$ or $(2,1,2 \cdots)$. Let $I_w$ be that
sequence. In this case, Lemma~\ref{decompdifference_lem} says that
\begin{equation} \label{ranktwocommute_eq}
\underbrace{\Del_1 \Del_2 \Del_1 \cdots}_{\text{$m$ times}} -
\underbrace{\Del_2 \Del_1 \Del_2 \cdots}_{\text{$m$ times}} = \sum_{w
  < w_0^{\al_1,\al_2}} \eta^{\al_1,\al_2}_w \Del_{I_w},\;\text{ where }\eta^{\al_1,\al_2}_w \in \sS.
\end{equation}
\end{example}

\begin{rem}
In \cite[Proposition 6.8]{HMSZ} several explicit formulas were given
for the coefficients $\eta_w^{\al_i,\al_j}$ ($i\neq j$) 
appearing in similar decompositions with coefficients on the right.
Moreover, case by case it was shown that 
$\eta^{\al_1,\al_2}_w\in \sS$ for root data of types $A_2$, $B_2$
and $D_2$ (the $G_2$ case was left open). 
The formula~\eqref{ranktwocommute_eq} provides a uniform proof of this fact for
all types including $G_2$.
\end{rem}

\begin{lem} \label{XIexpression_lem} If 
$\sS$ is $\RS$-regular, then for any sequence $I$, we can write $\Del_I=\sum_{v\in W}a_{I,I_v} X_{I_v}$ for some $a_{I,I_v}\in S$
such that:
\begin{enumerate}[\indent(1)]
\item \label{reducedcase_item} If $I$ is a reduced decomposition of $w\in W$, then  $a_{I,I_v}=0$ unless $v\le w$, and $a_{I,I_w}=1$.
\item \label{nonreducedcase_item} If $I$ is not reduced, then $a_{I,I_v}=0$ for all $v$ such that $\ell(v)\ge |I|$.
\end{enumerate}
and this decomposition is unique.
\end{lem}
\begin{proof} Uniqueness holds because the $(\Del_{I_w})_{w\in W}$ are linearly independent since they are linearly independent over $\qQ$, and the map $\sS\to \qQ$ is injective. 
\eqref{reducedcase_item} holds by Lemma~\ref{decompdifference_lem} with $I'=I_w$. If $I$ is not reduced, then we proceed by induction on its length, which is at least $2$. We write $I$ as $(i)\cup I'$. If $I'$ is not reduced we are done by induction and applying identity \eqref{Qcommute_eq} to move coefficients to the left of $\Del_i$. If $I'$ is a reduced sequence of $w'\in W$, then 
$$\Del_i \Del_{I'} \equalbyeq{reducedcase_item} \Del_i \sum_{v\leq w'} a_v \Del_{I_v} \equalbyeq{delqcomm_eq} \sum_{v<w'} \Del_i(a_v) \Del_{I_v} + \Del_i \Del_{I_{w'}} +\sum_{v<w'} a_v \Del_i \Del_{I_v}.$$
By induction, the third term is irrelevant. Since $(i)\cup I_{w'}$ is not reduced, there is a reduced sequence $I''$ of $w'$ with first term $i$ by the exchange property \cite[Ch. IV, \S 1, no 5, Prop. 4]{Bo68}. Using part \eqref{reducedcase_item}, and identity \eqref{square_eq}, $\Del_i \Del_{I_{w'}}$ also decomposes as a linear combination of $\Del_{I_v}$ with $v$ of length at most $\ell(I_{w'})=|I'|$ and we are done.
\end{proof}

As an immediate consequence we obtain the following generalization of
\cite[Lem.~6.13]{HMSZ} and \cite[Thm.~4.6.(a)]{KK}.

\begin{prop}\label{tDFcoincidence_prop}
If 
$\sS$ is $\RS$-regular, then the $R$-algebra $\DcF$ is free as a left $\sS$-submodule of $\QW$, with basis $(\Del_{I_w})_{w\in W}$.
\end{prop}
\begin{proof} 
We've already seen that the $\Del_{I_w}$ are linearly independent.
Let $P$ denote the left $\sS$-submodule of $\QW$ generated by $\{\Del_{I_w}\}_{w\in W}$. By definition we have $P\subseteq \DcF$.
It suffices to prove that $P$ is a $R$-subalgebra, which reduces to
showing that $\Del_i q \Del_{I_w} \in P$ for any $i\in [n]$, $w\in W$
and $q\in \sS$. Using \eqref{Qcommute_eq}, it reduces further
to showing that $\Del_i\Del_{I_w}\in P$, which holds by Lemma~\ref{XIexpression_lem}. 
\end{proof}

\begin{rem}\label{tDFcoinc_rem} 
Under the hypotheses of~\ref{DFgenmodule_prop}, Proposition~\ref{DFgenmodule_prop} and Proposition~\ref{tDFcoincidence_prop} imply that $\DcF=\tDcF$. 
\end{rem} 

The following theorem describes the twisted formal group algebra $\QW$
and the formal affine Demazure algebra
$\DcF$ in terms of generators and relations. It generalizes \cite[Theorem~6.14]{HMSZ}.
 
\begin{theo} \label{presentationDF_theo}
Let $\RS\hookrightarrow \cl^\vee$ be a root datum and let $F$ be a
formal group law over~$R$. Assume that the formal group algebra
$\sS=\RcF$ is $\RS$-regular.
Let $\qQ=\Qc$ denote the localization of $\sS$.
Given a set of simple roots $\{\al_1,\ldots,\al_n\}$ with
associated simple reflections 
$\{s_1,\ldots,s_n\}$, let $m_{i,j}$ denote the order of the product $s_i s_j$ in
the Weyl group. 

Then elements $q\in \qQ$ (resp. $q\in \sS$) and the Demazure elements
$\Del_{i}=\Del_{\al_i}$ satisfy the following relations ($\forall i,j\in[n]$)
\begin{align}
\label{delqcomm_eq} & \Del_i q = \De_i(q) + s_i(q) \Del_i,  & \\
\label{square_eq} & \Del_i^2= \kp_{\al_i} \Del_i, & \\
\label{twocomm_eq} & \underbrace{\Del_i \Del_j \Del_i \cdots}_{\text{$m_{i,j}$ times}} - \underbrace{\Del_j \Del_i \Del_j \cdots}_{\text{$m_{i,j}$ times}} = \sum_{w < w_0^{\al_i,\al_j}} \eta^{\al_i,\al_j}_w \Del_{I_w}, ~\eta^{\al_i,\al_j}_w\in S. &
\end{align}
Here $w_0^{\al_1,\al_2}$ is defined in Example~\ref{ranktwo_ex}. These relations together with the ring law in $\sS$ and the fact that the
$\Del_i$ are $R$-linear 
form a complete set of relations in the twisted formal
group algebra $\QW$ (resp. the formal affine Demazure algebra $\DcF$).
\end{theo}
\begin{proof}
The three relations hold by \eqref{Qcommute_eq}, \eqref{rankone_eq} and \eqref{ranktwocommute_eq}. The presentation of the twisted formal group algebra $\QW$ holds by the same proof as \cite[Theorem 6.14]{HMSZ}. 
The presentation of the formal affine Demazure algebra~$\DcF$ follows
 similarly, because of the fact that $\eta^{\al_i,\al_j}_w\in \sS$.
Note that in \cite{HMSZ}, there is a standing assumption that the base ring $R$ is a domain, but the proof of this presentation goes through as long as the $X_{I_w}$ form a $\qQ$-basis of $\QW$ (by Cor. \ref{Delbasis_cor}) and an $S$-basis of $\DcF$ (by Prop. \ref{tDFcoincidence_prop}).
\end{proof}

We now explain the relationship between the algebras $\DcF$, $\tDcF$, $\QW$ and $R$-linear endomorphisms of $\qQ$ and $\sS$. Following \cite[Definition 4.5]{CPZ} 
let $\ED$ denote the subalgebra
of $\End_R(\sS)$ generated by Demazure operators 
$\De_\al$ for all $\al\in \RS$ and left multiplications by elements in $\sS$. 

Let $\phi\colon \QW \to \End_R(\qQ)$ be the $R$-algebra homomorphism induced
by the left action of $\QW$ on $\qQ$. By definition the formal group
algebra $\sS$ acts on the left on both $R$-algebras, $\phi$ is $\sS$-linear, and
$\phi(\Del_\al)=\De_\al$ for each $\al\in \RS$.

We then have the following commutative diagram, where
$\mathrm{Stab}(\sS)$ is the $R$-subalgebra of $\End_R(\qQ)$ mapping
$\sS$ into itself, $\mathrm{res}_{\sS}$ is the restriction of an
endomorphism of $\qQ$ to $\sS$ and $\phi_{\DcF}$ is the restriction of
$\phi$ to $\DcF$
\begin{equation} \label{Endo_diag}
\begin{split}
\xymatrix{
\DcF \ar@{->>}[dd]_{\phi_{\DcF}} \ar@{}[r]|-{\textstyle \subseteq} & \tDcF \ar[d] \ar@{}[r]|-{\textstyle \subseteq}  & \QW \ar[d]^{\phi} \\
 & \mathrm{Stab}(\sS) \ar[d] \ar@{}[r]|-{\textstyle \subseteq} & \End_R(\qQ) \ar[d]^{\mathrm{res}_S} \\
\ED \ar@{}[r]|-{\textstyle \subseteq} & \End_R(\sS) \ar@{}[r]|-{\textstyle \subseteq} & \Hom_R(\sS,\qQ) 
}
\end{split}
\end{equation}
Observe that the image of $\phi_{\DcF}$ is $\ED$ since $R$-algebra generators of $\DcF$ map to $R$-algebra generators of $\ED$. 

\begin{theo} \label{DF=ED_theo}
In Diagram \eqref{Endo_diag}, if $\tor$ is regular in $R$ and $\sS$ is $\RS$-regular, then the
map $\mathrm{res_S} \circ \phi$ is injective and, therefore, $\phi\colon \QW \to \End_R(\qQ)$ is injective and the map $\phi_{\DcF}\colon \DcF \to \ED$ is an isomorphism.
\end{theo}
\begin{proof}
Since $\QW$ is a free $Q$-module, it injects into the respective
twisted formal group algebra over $R[1/\tor]$,
so we can assume that the torsion
index is invertible in $R$. 
Let $y$ be an element acting by zero on $\sS$. We can express it as
$y=\sum_{v \in W} a_v \Del_{I_v}$
by Corollary~\ref{Delbasis_cor}. 
Applying $y$ successively to all $\De_{I_w}(u_0)$, $w\in W$, we
conclude that $a_v =0$ for each $v \in W$ by Lemma \ref{Deltamatrixinvert_lem}. This shows that $\mathrm{res_S}\circ \phi$ is injective, and the other two facts follow by diagram chase.
\end{proof}

\begin{rem} 
Note that by \cite[Theorem 4.11]{CPZ}, once reduced sequences $I_w$
have been chosen for all $w \in W$, the family $\{\De_{I_w}\}_{w\in
  W}$ forms a basis of $\ED$ as a $\sS$-module, so the $\sS$-module
$P\subseteq \DcF$ 
from the proof of Proposition~\ref{tDFcoincidence_prop} already surjects to $\ED$ via $\phi_{\DcF}$, even if $\tor$ is not regular in $R$.
\end{rem}

\section{Coproduct on the twisted formal group algebra}\label{sec:copr}

In the present section we introduce and study various (co)products on
the twisted formal group algebra $\QW$. We follow the ideas of \cite[(4.14)]{KK}
and \cite[(2.14)]{kk}. In this section we don't assume that $\sS$ is $\Sigma$-regular.

\medskip

Let $W$ be a group and let $R$ be a commutative ring.
The group ring $R[W]$ is a Hopf $R$-algebra endowed with
a cocommutative coproduct $\RWcopr\colon R[W]\to R[W]\ot
R[W]$ defined by $\de_w \mapsto \de_w \ot \de_w$. Its counit $\RWcoun$ sends any $\de_w$ to $1$
(unless otherwise specified, all modules and tensor products are
over $R$). 

Assume that $W$ acts on $Q$ by means of $R$-linear automorphisms.
Let $\QW=Q\ot R[W]$ 
be the twisted formal group algebra considered in $\S5$.
Let $\incQtoQW\colon \qQ \to \QW$ and $\incRWtoQW\colon R[W]\to \QW$
denote the natural inclusions of $R$-algebras defined by $q\mapsto q \ot
1$ and $\de_w\mapsto 1 \ot \de_w$ respectively. 

\begin{defi}
We define a
coproduct
$\QWRcopr\colon \QW \to \QW\ot \QW$ to be the composition of 
$R$-linear morphisms
\[
\qQ\ot R[W] \tooby{\id_\qQ \ot \RWcopr} \qQ \ot R[W] \ot R[W]
\tooby{\id_{\qQ\ot R[W]}\ot \incRWtoQW} (\qQ \ot R[W]) \ot
(\qQ \ot R[W]).
\] 
\end{defi}

We endow the $R$-module $\QW \ot \QW$ with the usual tensor product
algebra structure, i.e. the product is given by 
$(z_1 \ot z'_1)\cdot (z_2 \ot z'_2) \mapsto (z_1 z_2 \ot z'_1 z'_2)$.  

\begin{lem}
The coproduct $\QWRcopr$ is associative and is a morphism of
$R$-algebras, i.e. $\QWRcopr(z\cdot z') = \QWRcopr(z)\cdot\QWRcopr(z')$.
\end{lem}

\begin{proof}
By linearity, it suffices to check these properties on elements of the
form $q\de_w$ with $q\in \qQ$ and $w \in W$. 
By definition we have $\QWRcopr(q\de_w)=q\de_w\ot
\de_w$
Therefore, we obtain
\[
\QWRcopr(q\de_w \cdot
q'\de_{w'})=\QWRcopr(qw(q')\de_{ww'})=
qw(q')\de_{ww'}\ot \de_{ww'}=
\]
\[
q\de_w \cdot
q'\de_{w'}\ot \de_w\cdot \de_{w'}=\QWRcopr(q\de_w)\cdot \QWRcopr(q'\de_{w'}).\qedhere
\]
\end{proof}

\begin{rem}
The coproduct $\QWRcopr$ has no counit in general, and it is not cocommutative, because coefficients in $\qQ$ are arbitrarily put on the left.
\end{rem}
 
\begin{lem}
The map $\incRWtoQW$ is a map of $R$-coalgebras (without counits), i.e. $\QWRcopr \circ \incRWtoQW = (\incRWtoQW \ot \incRWtoQW) \circ \RWcopr$.
\end{lem}
\begin{proof}
It is straightforward on $R$-basis elements $\de_w$ of $R[W]$.
\end{proof}

\begin{lem}
The map $\id_{\qQ} \ot \RWcopr\colon  \QW \to \QW \ot R[W]$ is a morphism of $R$-algebras (with units), and so is the coproduct $\QWRcopr$. 
\end{lem}
\begin{proof}
The first claim can be checked directly on elements $q\de_w$, and 
the second one follows since the second map in the composition defining $\QWRcopr$ is a morphism of $R$-algebras.
\end{proof}

\begin{rem} Observe that
the product in $\QW$ is, therefore, a morphism of $R$-coalgebras (without units), since it is the same condition. Similarly, the unit of $\QW$ preserves coproducts. 
\end{rem}

We now define a coproduct on $\QW$ viewed as a \emph{left} $\qQ$-module. Beware that $\qQ$ is not central in $\QW$, so $\QW$ is not a $\qQ$-algebra in the usual sense. The convention is still that unlabeled tensor products are over $R$. 

Note that the tensor product $\QW \ot_{\qQ} \QW$ of left
$\qQ$-modules has a natural structure of left $\qQ$-module via action $q\cdot (z\ot z')= qz\ot z'=z\ot q z'$.

Let $\Qproj:\QW \ot \QW \to \QW\ot_\qQ \QW$ be the projection map. It is a morphism of left $\qQ$-modules for both actions of $\qQ$ on the source, on the left or right factor. The $\qQ$-submodule $\Qker$ is generated by elements of the form $qz\ot z' -z \ot qz' =(q\ot 1 -1 \ot q) \cdot (z\ot z')$.
Let $\Qinc=\id_{\QW} \ot \incRWtoQW$. The morphism of left
$\qQ$-modules 
$$
\Qiso = \Qproj \circ \Qinc\colon \QW \ot R[W] \to \QW
\ot_{\qQ} \QW
$$ 
is an isomorphism. 

\begin{defi} \label{productQWQW_defi}
We define the product $\QWprod$ on $\QW \ot_\qQ \QW$ by transporting the product of the tensor product algebra $\QW \ot R[W]$ through the isomorphism $\Qiso$, i.e. 
$$z\odot z':=\Qiso(\Qiso^{-1}(z)\cdot \Qiso^{-1}(z'))\qquad \text{for $z,z'$ in $\QW\ot_\qQ \QW$.}$$
\end{defi}

\begin{rem}
Note that the morphism $\Qiso$ is a morphism of $R$-algebras, and so is $\Qinc$, but $\Qproj$ does \emph{not} respect products in general. We therefore have an inclusion $\Qinc$ of $R$-algebras, that is split by $\Qiso^{-1}\circ \Qproj$, a map of $\qQ$-modules, but not a map of $R$-algebras. Still, $\Qproj$ restricted to $\im(\Qinc)$ is a morphism of $R$-algebras, because it is the inverse of the morphism of $R$-algebras $\Qinc$.
\end{rem}

\begin{defi}
We define a coproduct $\QWcopr\colon \QW \to \QW \ot_{\qQ} \QW$ on $\QW$ viewed as a $\qQ$-module to be the morphism of left $\qQ$-modules given by the composition $\Qproj \circ \QWRcopr=\Qiso \circ (\id_{\qQ} \ot \RWcopr)$.
\end{defi}
We therefore have $\QWcopr(q \delta_w)=q \delta_w \otimes \delta_w= \delta_w \otimes q\delta_w$.

\begin{prop}\label{QWcocomalg_prop}
The coproduct $\QWcopr$ is $\qQ$-linear, associative and
cocommutative, with counit $\QWcoun:\QW \to \qQ$ defined by $x \mapsto x(1)$ (action of $\QW$ on $\qQ$ defined in section \ref{sec:invar}).
Therefore, $\QW$ is a cocommutative coalgebra in the category of left $\qQ$-modules
\end{prop}
\begin{proof}
Since all maps involved are morphisms of left $\qQ$-modules, it suffices to check it on $\qQ$-basis elements $\de_w$ of $\QW$. All three properties are then straightforward.
\end{proof}

\begin{rem} \label{counitequalsaction_rem}
The counit $\QWcoun$ satisfies that for any $q\in \qQ$ and any $x\in \QW$, we have $\QWcoun(xq)=x(q)$.
\end{rem}

\begin{lem}
The coproduct $\QWcopr$ is a morphism of $R$-algebras.
\end{lem}

\begin{proof}
We have $\QWcopr=\Qiso \circ (\id_{\qQ} \ot \RWcopr)$ and both maps are maps of $R$-algebras.
\end{proof}

\begin{lem} \label{kerleftstable_lem}
The kernel $\Qker$ is a right ideal, and is also stable by multiplication on the left by elements in the image of $\QWRcopr$. In particular $\Qker \cap \im(\QWRcopr)$ is a double-sided ideal of the $R$-subalgebra $\im(\QWRcopr)$.
\end{lem}
\begin{proof}
It is a right ideal, since it is additively generated by elements of the form $(q\ot 1 -1 \ot q)\cdot (z\ot z')$, with $q\in \qQ$ and $z,z'\in \QW$. To prove the second part of the claim, it suffices to prove that any element of the form $\QWRcopr(x)\cdot (q\ot 1-1\ot q)$ with $q \in \qQ$ and $x\in \QW$ is a sum of elements of the form $(q'\ot 1-1\ot q)\cdot \QWRcopr(x)$ for some $q'\in \qQ$ and $x'\in \QW$. It is enough to check it when $x=q_w\de_w$, in which case one has $q'=w(q)$.
\end{proof}
\begin{rem}
Note that this left stability uses the cocommutativity of $\RWcopr$ in an essential way.
\end{rem}

\begin{cor} \label{qprojbar_cor}
The $R$-submodule $\Qim+\Qker \subseteq \QW \ot \QW$ is an $R$-subalgebra and the morphism $\Qproj$ restricted to $\Qim+\Qker$ is a morphism of $R$-algebras. 
\end{cor} 

\begin{prop}
For elements $y\in \QW \ot R[W]$ and $z \in \QW \ot \QW$, we have 
\begin{equation} \label{projright_eq}
\Qproj(z) \odot \Qiso(y) = \Qproj(z\cdot \Qinc(y))
\end{equation}
and if $ \Qinc(y)\cdot \Qker\subseteq \Qker$, we also have 
\begin{equation} \label{projleft_eq}
\Qiso(y) \odot \Qproj(z) = \Qproj(\Qinc(y) \cdot z).
\end{equation}
In particular, for any $x\in \QW$, we have
\begin{equation} \label{coprproj_eq}
\Qproj(z) \odot \QWcopr (x) = \Qproj\big(z\cdot \QWRcopr(x)\big) \quad\text{and}\quad \QWcopr(x) \odot \Qproj(z) = \Qproj\big(\QWRcopr(x)  \cdot z\big).
\end{equation}
\end{prop}
\begin{proof}
We have
\begin{multline*}
\Qproj(z)\odot \Qiso(y) = \Qiso(\Qiso^{-1}\Qproj(z)\cdot y) = \Qproj \Qinc(\Qiso^{-1}\Qproj(z)\cdot y) \\ 
= \Qproj(\Qinc \Qiso^{-1}\Qproj(z)\cdot \Qinc(y)) = \Qproj((z+\Qker)\cdot \Qinc(y)) = \Qproj(z\cdot \Qinc(y) + \Qker) = \Qproj(z\cdot \Qinc(y)).
\end{multline*}
Here we have used that $\Qker$ is a right ideal. Similarly,
reversing the product we obtain  \eqref{projleft_eq}.
For identity \eqref{coprproj_eq}, just set $y=(\id_\qQ \ot \RWcopr)(x)$ and use Lemma~\ref{kerleftstable_lem} to satisfy the extra assumption. 
\end{proof}

\section{Coproduct on the formal affine Demazure algebra}\label{sec:copraffine}

In the present section we define the coproduct
on the formal affine Demazure algebra $\DcF$ by restricting
the coproduct $\QWcopr$ on $\QW$ (see Theorem~\ref{Demcoprod_theo}) and
compute it for Demazure elements (see
Proposition~\ref{coproductXI_prop}).

\medskip

All results of the present section assume that 
the formal group algebra $\sS$ is $\RS$-regular (Definition~\ref{asmreg_defi}).

\begin{lem} \label{DDinjQWQW_lem}  The induced map $\DcF \ot_\sS \DcF
  \to \QW \ot_{\qQ} \QW$ of left $\sS$-modules is injective.
\end{lem}
\begin{proof}
The $\sS$-module $\DcF\ot_\sS \DcF$ is free with basis
$(\Del_{I_v} \ot \Del_{I_w})_{(v,w)\in W^2}$, by
Proposition~\ref{tDFcoincidence_prop}. On the other hand this basis is
also a $\qQ$-basis of $\QW \ot_{\qQ} \QW$ by Corollary~\ref{Delbasis_cor}. Since $\sS$ is $\Sigma$-regular, by
Lemma~\ref{localizationinj_lem}, $\sS$ injects in $\qQ$ and we are done.
\end{proof}

We can therefore identify $\DcF \ot_\sS \DcF$ with the
$\sS$-submodule image $\Qproj(\DcF\ot \DcF)$ in $\QW \ot_\qQ
\QW$. 
Beware, however, that through this identification, the product $\odot$
does \emph{not} correspond to the usual product of $\DcF \ot
\DcF$, and 
$\DcF \ot_\sS \DcF$ is \emph{not} stable by the product $\odot$.

By direct computation
we obtain in $\QW \ot \QW$ that $\QWRcopr(\Del_\al) =$
\begin{equation} \label{DelalphaQWQW_eq}
\begin{split}
& \Del_\al \ot 1 + 1\ot \Del_\al - x_\al \Del_\al \ot \Del_\al + (\frac{1}{x_{\al}} \ot 1 - 1 \ot \frac{1}{x_{\al}})(\de_{s_\al} \ot 1 - \de_{s_\al} \ot \de_{s_\al}) \\
& = \Del_\al \ot 1 + 1\ot \Del_\al - \Del_\al \ot x_\al \Del_\al  + (\frac{1}{x_\al} \ot 1 - 1 \ot \frac{1}{x_\al})(1 \ot 1 - 1 \ot \de_{s_\al}).
\end{split}
\end{equation}
It implies that in $\QW \ot_\qQ \QW$ we have
\begin{equation} \label{DelalphaQWQQW_eq}
\QWcopr(\Del_\al)=\Del_\al \ot 1 + 1\ot \Del_\al - x_\al \Del_\al \ot \Del_\al.
\end{equation}
 
\begin{theo}\label{Demcoprod_theo} 
The coproduct $\QWcopr\colon \QW \to \QW \ot_{\qQ} \QW$ restricts to an $\sS$-linear coproduct $\QWcopr\colon \DcF \to \DcF \ot_\sS \DcF$ with counit $\Dcoun \colon \DcF\to \sS$ obtained by restricting the counit $\QWcoun \colon \QW \to \qQ$. One has $\Dcoun(\Del_{I_w})=\delta_{w,1}$ (the Kronecker symbol).
\end{theo}

\begin{proof}
Since $\QWcopr$ is $\qQ$-linear and is a morphism of $R$-algebras, it suffices to show that $\QWcopr(\Del_\al)\odot \Qproj(z) \in \Qproj(\DcF \ot \DcF)$ for any root $\al$ and any $z \in \DcF \ot \DcF$ (in particular for $z=1\ot 1$). 
We thus compute 
\[
\begin{split}
\QWcopr(\Del_\al)\odot \Qproj(z) & \equalbyeq{coprproj_eq} \Qproj(\QWRcopr(\Del_\al)\cdot z) \\
 & \hspace{1ex}= \Qproj\big((1 \ot \Del_\al + \Del_\al \ot 1 -x_\al \Del_\al\ot \Del_\al+ \Qker)\cdot z \big) \\
& \hspace{-2.2ex}\equalby{Lem. \ref{kerleftstable_lem}} \Qproj\big((1 \ot \Del_\al + \De_\al\ot 1- x_\al \Del_\al \ot \Del_\al)\cdot z + \Qker\big) \\
& \hspace{1ex}= \Qproj\big((1 \ot \Del_\al + \De_\al\ot 1-x_\al \Del_\al \ot \Del_\al)\cdot z)
\end{split}
\]
and the last term is in $\Qproj(\DcF \ot \DcF)$, since the product is
term by term in $\QW \ot \QW$.

The counit $\QWcoun$ on $\QW$ obviously maps $\DcF$ to $\sS$ by remark \ref{counitequalsaction_rem} since $\DcF \subseteq \tDcF$. One has $\Dcoun(\Del_{I_w})=\De_{I_w}(1)=\delta_{w,1}$.
\end{proof}

Given a sequence $I$ in $[n]$ we want to provide a formula for the
coproduct of the element $\Del_I$. 
Relation \eqref{Qcommute_eq} 
\begin{equation} \label{commuteDelq_eq}
\Del_\al \cdot q = \De_\al(q) + s_\al(q)\cdot
\Del_\al,\quad \text{ for all }q\in \qQ\text{ and }\al\in \RS.
\end{equation}
generalizes as:
\begin{lem}\label{commut_rel}
For any sequence $I$ and for any $q \in Q$, we have
\begin{equation} 
\Del_I \cdot q = \sum_{E \subseteq [l]} \phi_{I,E}(q)\cdot \Del_{I_{|E}}
\end{equation}
where $\phi_{I,E}=B_1\circ \cdots \circ B_l$ with
$$B_j=
\begin{cases}
B_{i_j}^{(0)}, & \text{if $j\in E$}, \\
B_{i_j}^{(-1)} & \text{otherwise.}
\end{cases}$$
\end{lem}
\begin{proof}
Induction using \eqref{commuteDelq_eq}.
\end{proof}

\begin{lem} \label{prodformDD_lem}
For any sequence $I=(i_1,\ldots,i_l)$  in $[n]$ we have in $\QW \ot \QW$
\begin{equation} \label{prodformDelI_eq}
\prod_{k=1}^l (\Del_{i_k} \ot 1 + 1 \ot \Del_{i_k} - x_{i_k} \Del_{i_k} \ot \Del_{i_k}) =\sum_{E_1,E_2\subset [l]} (p_{E_1,E_2}^I\Del_{I_{|E_1}}) \ot \Del_{I_{|E_2}}
\end{equation}
where the element $p_{E_1,E_2}^I\in \sS$ is defined in Lemma~\ref{productform_lem}.
\end{lem} 
\begin{proof}
We prove the formula by induction on the length $l$ of $I$. For $l=1$, it holds by identity \eqref{DelalphaQWQQW_eq}. Let $I'=(i_2,\ldots,i_l)$. For fixed $E_1,E_2 \subseteq [l-1]$, and for $p \in \sS$ we have
\[
\begin{split}
& (1\ot \Del_{i_1} +\Del_{i_1} \ot 1 -x_{i_1} \Del_{i_1}\ot \Del_{i_1})\cdot \big((p\Del_{I'_{|E_1}})\ot \Del_{I'_{|E_2}}\big) \\ 
& = (p\Del_{I'_{|E_1}})\ot (\Del_{i_1}\Del_{I'_{|E_2}}) +(\Del_{i_1}p\Del_{I'_{|E_1}})\ot \Del_{I'_{|E_2}}-(x_{i_1}\Del_{i_1}p\Del_{I'_{|E_1}})\ot (\Del_{i_1}\Del_{I'_{|E_2}}) \\
& \hspace{-1ex}\equalbyeq{commuteDelq_eq} (p\Del_{I'_{|E_1}})\ot (\Del_{i_1}\Del_{I'_{|E_2}})+(\De_{i_1}(p)\Del_{I'_{|E_1}})\ot \Del_{I'_{|E_2}} + (s_{i_1}(p)\Del_{i_1}\Del_{I'_{|E_1}})\ot \Del_{I'_{|E_2}} \\ 
& \hspace{2ex}-(x_{i_1}\De_{i_1}(p)\Del_{I'_{|E_1}})\ot (\Del_{i_1}\Del_{I'_{|E_2}}) -(x_{i_1}s_{i_1}(p)\Del_{i_1}\Del_{I'_{|E_1}})\ot (\Del_{i_1}\Del_{I'_{|E_2}}) \\
& =(\De_{i_1}(p)\Del_{I'_{|E_1}})\ot \Del_{I'_{|E_2}}+ (s_{i_1}(p)\Del_{i_1}\Del_{I'_{|E_1}})\ot \Del_{I'_{|E_2}} \\ 
& \hspace{2ex}+ (s_{i_1}(p)\Del_{i_1}\Del_{I'_{|E_1}})\ot \Del_{I'_{|E_2}} -(x_{i_1}s_{i_1}(p)\Del_{i_1}\Del_{I'_{|E_1}})\ot (\Del_{i_1}\Del_{I'_{|E_2}}) 
\end{split}
\]
and the result follows by induction as in the end of the proof of Lemma~\ref{productform_lem}.
\end{proof}
 
\begin{prop}  [cf. Proposition 4.15 \cite{KK}] \label{coproductXI_prop}
For any $I=(i_1,\ldots,i_l)$ we have
\begin{equation} \label{coproductXI_eq}
\QWcopr(\Del_I)=\sum_{E_1,E_2\subset [l]} p_{E_1,E_2}^I\Del_{I_{|E_1}} \ot \Del_{I_{|E_2}}. 
\end{equation}
\end{prop}
\begin{proof}
We have
\[
\begin{split}
& \QWcopr(\Del_I)=\QWcopr(\Del_{i_1})\odot\cdots\odot \QWcopr(\Del_{i_l}) = \Qproj \QWRcopr(\Del_{i_1})\odot \cdots \odot \Qproj \QWRcopr(\Del_{i_l}) \\
& \equalbyeq{DelalphaQWQQW_eq} \Qproj (1\ot \Del_{i_1} +\Del_{i_1} \ot 1 - x_{i_1} \Del_{i_1} \ot \Del_{i_1}) \odot \cdots \odot \Qproj (1\ot \Del_{i_l} +\Del_{i_l} \ot 1 - x_{i_l} \Del_{i_l} \ot \Del_{i_l}) \\
& \equalby{Cor. \ref{qprojbar_cor}} \Qproj\big((1\ot \Del_{i_1} +\Del_{i_1} \ot 1 - x_{i_1} \Del_{i_1} \ot \Del_{i_1}) \cdots (1\ot \Del_{i_l} +\Del_{i_l} \ot 1 - x_{i_l} \Del_{i_l} \ot \Del_{i_l}) \big) \\
& \equalbyeq{prodformDelI_eq}\sum_{E_1,E_2\subset [l]} p_{E_1,E_2}^I\Del_{I_{|E_1}} \ot \Del_{I_{|E_2}}.\mbox{\qedhere} 
\end{split} 
\]
\end{proof}
\begin{rem}
Note that the only subtlety in the proof of the previous formula is to check that products of elements $\QWcopr(\Del_i)$ can indeed be computed term-wise in the tensor product, which is the content of Corollary~\ref{qprojbar_cor}.
\end{rem}
 
\section{Coproduct on the algebra of Demazure operators}\label{sec:compare}

In the present section we lift the coproduct on $\eD$ defined in \cite{CPZ} to a natural coproduct on $\ED$, and show it is isomorphic to the coproduct $\QWcopr$ on $\DcF$ through the isomorphism $\phi_{\DcF}$ of Theorem \ref{DF=ED_theo}.
\medskip

Unless otherwise specified, all tensor products are still over $R$, all $\Hom$ and $\End$ groups are homomorphisms of $R$-modules. When $M$ and $N$ are topological $R$-modules, let $\Homc(N,M)$ (resp. $\Endc(M)$) denote the $R$-submodule of $\Hom(N,M)$ (resp. of $\End(M)$) of continuous homomorphisms. When $M$ is a topological module over $\sS$, and $R$ acts through $S$, $\Hom(N,M)$ has an obvious structure of left $\sS$-module, by $s\cdot f= (x \mapsto sf(x))$. Then $\Homc(N,M)$ is an $\sS$-submodule of $\Hom(N,M)$, and in particular $\Endc(\sS)$ is an $\sS$-submodule of $\End(\sS)$. By construction, $\ED$ is an $\sS$-submodule of $\Endc(\sS)$. Let $\Inc_1: \ED \to \Endc(\sS)$ denote the inclusion.

Let $\sS \otimes \sS$ (resp. $\sS \otimes \sS \otimes \sS$) be endowed with the topology defined by the $\IF^i\otimes \IF^j$, $i,j\in \N$ (by the $\IF^i\otimes \IF^j \otimes \IF^k$, $i,j,k\in \N$) as a basis of neighborhoods of zero. 
Let 
$$\begin{array}{rclc}
\Inc_2: & \ED \otimes_S \ED & \to & \Homc(\sS \otimes \sS,\sS) \\
 & D_1 \otimes D_2 & \mapsto & ((r\otimes s) \mapsto D_1(r)D_2(s))
\end{array}$$ 
and similarly
$$\begin{array}{rclc}
\Inc_3: & \ED \otimes_S \ED \otimes_S \ED & \to & \Homc(\sS \otimes \sS \otimes \sS,\sS) \\
 & D_1 \otimes D_2 \otimes D_3 & \mapsto & ((r\otimes s \otimes t) \mapsto D_1(r)D_2(s)D_3(t)).
\end{array}$$ 

\begin{lem} \label{fundInc_lem}
If the torsion index $\tor$ is regular in $R$, the maps $\Inc_1, \Inc_2$ and $\Inc_3$ are injective.
\end{lem}
\begin{proof}
By definition, $\Inc_1$ is an inclusion. For $\Inc_2$, by \cite[Theorem 4.11]{CPZ}, once reduced sequences $I_w$ have been chosen for all $w \in W$, the family $\{\De_{I_w}\}_{w\in W}$ forms a basis of $\ED$ as an $\sS$-module. Let us assume that $\sum_{v,w} a_{v,w} \De_{I_v} \otimes \De_{I_w}$ is mapped to the zero morphism. Thus, for any $r,s \in \sS$, we have $\sum_w \sum_v a_{v,w} \De_{I_v}(r)\De_{I_w}(s)=0$. Replacing successively $s$ by $\De_{I_{w'}}(u_0)$ for all $w'\in W$, we obtain that for any $w'$, $\sum_w \zeta_w(r) \De_{I_w}(\De_{I_{w'}}(u_0))=0$ with $\zeta_w(r)=\sum_v a_{v,w} \De_{I_v}(r)$. So, given $w$, for any $r \in \sS$,  we have $\zeta_w(r)=0$ by Lemma \ref{Deltamatrixinvert_lem}. Replacing $r$ successively by $\De_{I_{v'}}(u_0)$ for all $v'\in W$, we similarly obtain $a_{v,w}=0$. 

The proof for $\Inc_3$ is the same, but with one extra step for the last factor.
\end{proof}
Another way of phrasing the previous lemma is that an element of $\ED\otimes_\sS \ED$ is characterized by how it acts on $\sS\otimes \sS$ (with values in $\sS$), and similarly with three factors. 

\begin{lem} \label{augmDaction_lem}
If $\tor$ is regular in $R$, then the preimage by $\Inc_1$ of $\Homc(\sS,\IF)$ is $\IF \ED$.
\end{lem}
\begin{proof}
Let $d \in \ED$ such that $d(s) \in \IF$ for any $s\in \sS$. Decompose $d=\sum_w a_w \De_{I_w}$ by \cite[Theorem 4.11]{CPZ}. After inverting the torsion index in $R$, by Lemma \ref{Deltamatrixinvert_lem}, we obtain that for any $w \in W$, the element $a_w$ is in $\IFR{R[1/\tor]}$, the augmentation ideal of $\FGR{R[\tfrac{1}{\tor}]}{\cl}{F}$. Since $\IFR{R[1/\tor]} \cap \sS=\IF$ inside $\FGR{R[\tfrac{1}{\tor}]}{\cl}{F}$, we are done. 
\end{proof}

Let $\homcopr: \Endc(\sS) \to \Homc(\sS\otimes \sS,\sS)$ be the morphism of left $\sS$-modules sending $f \in \Endc(\sS)$ to $((u\otimes v) \mapsto f(uv))$. 

\begin{theo} \label{EDCopr_theo}
Assume that $\tor$ is regular. Along $\Inc_2$, the morphism $\homcopr$ factors (uniquely) through a morphism of left $\sS$-modules $\EDcopr: \ED \to \ED\otimes_\sS \ED$, which is thus the unique such morphism satisfying the formula 
\begin{equation} \label{coprodprod_eq}
\Inc_2(\EDcopr(f))(u\otimes v) = f(uv).
\end{equation}
It is an associative and commutative $\sS$-linear coproduct, with counit $\EDcoun$ mapping $f$ to $f(1)$ (evaluation at $1$).
\end{theo} 
\begin{proof}
Since $\ED \otimes_\sS \ED$ injects by $\Inc_2$ in $\Homc(\sS\otimes \sS,\sS)$, uniqueness is clear. Now $\homcopr$ restricts by \cite[Lemma 7.1 (1)]{CPZ}. 
Associativity follows from the fact that both compositions involved $\ED \to \ED \otimes_S \ED \otimes_S \ED$ composed with $\Inc_3$ send $D \in \ED$ to $((u,v,w) \mapsto D(uvw))$, so they are equal by injectivity of $\Inc_3$. Commutativity follows similarly with $\Inc_2$ instead of $\Inc_3$, and the fact that $\EDcoun$ is a (left) counit holds because $(\EDcoun \otimes \id_S)\circ \EDcopr(f)$ sends $f$ to $(v\mapsto f(1\cdot v)=f(v))$. It is also a right counit by the same argument on the right side.
\end{proof}

\begin{theo}\label{DF=DLambda_theo}
When the torsion index $\tor$ is regular in $R$ and $\sS$ is $\RS$-regular, the map $\phi_{\DcF}: \DcF \to \ED$ of Theorem \ref{DF=ED_theo} is an isomorphism of $\sS$-coalgebras with counits.
\end{theo}
\begin{proof}
Since $\tor$ is regular, the $(\Del_{I_w})_{w\in W}$ are an $\sS$-basis of $\DcF$. The $(\De_{I_w})_{w \in W}$ are an $\sS$-basis of $\ED$ in any case. The identification of $\QWcopr$ and $\EDcopr$ through $\phi_{\DcF}$ then follows from formulas \eqref{coproductXI_eq} and \eqref{productform_eq}, via the uniqueness of Theorem \ref{EDCopr_theo}. 
\end{proof}

The coproduct that appears in \cite{CPZ} is then constructed as follows. The augmentation map $\aug:\RcF\to R$ induces a map of $R$-modules $\aug_*\colon \ED \to \Hom_R(S,R)$ given by $f\mapsto \aug \circ f$. Let $\eD$ be the image of $\ED$ by this map. 
As in the proof of \cite[Thm.~7.3]{CPZ} 
we define an $R$-linear coproduct 
\[
\TE \colon\eD \to \eD \otimes_R \eD
\]
by
\[
\TE (f)(u\otimes v):=f(uv),\quad f\in \eD,\; u,v\in\RcF.
\]
Here the property (which follows from \cite[Prop.~3.8]{CPZ})
\[
\aug\De_{i}(u\cdot v)=\aug\De_{i}(u)\cdot \aug(v)+\aug(u)\cdot
\aug\De_{i}(v),\quad u,v\in \RcF
\]
guarantees that the coproduct $\TE$ is well defined (cf. \cite[Lem.~7.1]{CPZ}) and 
\[
\TE(\aug\De_{i})=\aug\De_{i}\otimes \aug+\aug\otimes \aug\De_{i}.
\]

Through $\aug:\sS \to R$, we can restrict $\eD$ to an $\sS$-module. Then $\eD\otimes_R \eD = \eD \otimes_S \eD$, and $\TE$ is an $\sS$-linear coproduct. The morphism $\aug_*:\ED \to \eD$ is clearly a morphism of $\sS$-coalgebras. In other words, the diagram
\begin{equation}\label{EDeDcopr_diag}
\begin{split}
\xymatrix{
\DcF \ar[r]^{\phi_{\DcF}} \ar[d]^{\QWcopr} & \ED \ar[r]^{\aug_*} \ar[d]^{\EDcopr} & \eD \ar[d]^{\TE} \\
\DcF \otimes_\sS \DcF \ar[r]^-{\phi_{\DcF}^{\otimes 2}} & \ED\otimes_S \ED \ar[r]^{\aug_*^{\otimes 2}} & \eD \otimes_R \eD
}
\end{split}
\end{equation}
commutes.
 
\begin{rem}
The significance of the coproduct $\TE$ follows from the following fact:

Following \cite[\S6]{CPZ} consider the dual
$\eD^*=\Hom_R(\eD,R)$ (in loc. cit. $\eD^*=\mathcal{H}(M)_F$).  
 Then by \cite[Thm.~7.3 and Thm.~13.12]{CPZ} the coproduct $\TE$ on $\eD$
  induces (by duality) an $R$-algebra structure on $\eD^*$. Moreover,
under the hypotheses of \cite[Thm.~13.3]{CPZ}
there is a natural ring isomorphism $\eD^*\simeq \hh(G/B)$, where
$\hh(G/B)$ is the algebraic oriented cohomology of the variety of
complete flags $G/B$.
\end{rem}

\section{The dual of a formal affine Demazure algebra}\label{sec:dualaffine}

In the present section we consider the dual $(\ED)^\star$ of the algebra of Demazure operators, or equivalently the dual of the formal
affine Demazure algebra $\DcF^\star$ and prove Theorem~\ref{mainthmdual_theo}.

\medskip

We will need the following classical facts, that the reader can easily verify anyway.

Let $S$ be a commutative unital ring and let $D$ be a free left $S$-module of finite type.
Assume that $D$ is an $\sS$-coalgebra, with linear coproduct 
$\QWcopr \colon D \to D\ot_S D$ and counit $\Dcoun:D\to S$. 
Then the dual $D^\star=\Hom_S(D,S)$ has a natural structure of $S$-algebra with unit, given by dualizing the coproduct and using the natural isomorphism $D^\star\otimes_S D^\star \simeq (D\otimes_S D)^\star$, and dualizing the counit, i.e. $\Dcoun$, seen as an element of $D^\star$, is the unit. It is commutative if and only if $D$ is cocommutative. Let $p^\theta_{\theta_1,\theta_2}\in S$ be the coefficients of the coproduct on a basis $\{x_\theta\}_{\theta\in \Theta}$ of $D$ ($\Theta$ is therefore finite), i.e. for every $\theta\in \Theta$, 
$$\QWcopr(x_\theta)=\sum_{\theta_1,\theta_2\in \Theta} p^\theta_{\theta_1,\theta_2}x_{\theta_1}\ot x_{\theta_2}.$$
On the dual basis $(x_\theta^\star)_{\theta \in \Theta}$ of $(x_\theta)_{\theta\in \Theta}$, the product of $D^\star$ is given by 
$$x^\star_{\theta_1}x^\star_{\theta_2}=\sum_{\theta\in \Theta} p^\theta_{\theta_1,\theta_2} x^\star_\theta.$$

\medskip

\emph{For the rest of this section, we assume that $2\tor$ is regular in $R$. It implies that $\sS$ is $\RS$-regular by Lemma \ref{rootinj_lem}.}
\medskip

We apply the above construction to $D=\ED$, with its coalgebra structure given by Theorem \ref{EDCopr_theo}. This yields a commutative algebra with unit $\ED^\star$. Given choices of a reduced decompositions $I_w$ for each element $w \in W$, we have a basis $(\De_{I_w})_{w\in W}$ of $\ED$. 

\begin{rem}
Instead of $\ED$, we can use $\DcF$, with its coalgebra structure given by Theorem \ref{Demcoprod_theo} and its basis $(\Del_{I_w})_{w\in W}$. Since we have $\DcF\simeq \ED$ as $\sS$-coalgebras by Theorem \ref{DF=DLambda_theo}, we dually have $(\ED)^\star\simeq \DcF^\star$ as algebras, and it won't make any difference. In particular, Theorem \ref{mainthmdual_theo} also holds for $\DcF$. 
\end{rem}

The $R$-dual $(\eD)^*=\Hom_R(\eD,R)$ is an $R$-algebra with unit by dualizing the coproduct $\TE$ of section \ref{sec:compare}. By \cite[Prop. 5.4]{CPZ}, the $\aug \De_{I_w}$, $w \in W$, form a basis of $\eD$. Using $\aug:\sS \to R$, we can restrict $(\eD)^*$ to an $\sS$-algebra with unit.

\begin{lem}
The map $\xi:(\ED)^\star \to (\eD)^*$ sending $f$ to the map $\aug d \mapsto \aug f(d)$ is well defined, and it is a surjective map of $\sS$-algebras with units.
\end{lem}
\begin{proof}
If $\aug d_1 = \aug d_2$, as elements of $\Homc(\sS,R)$, or in other words $(d_1-d_2) \in \Homc(\sS,\IF)$, then Lemma \ref{augmDaction_lem} shows that $d_1-d_2 \in \IF \ED$, and thus $\aug f(d_1)=\aug f(d_2)$ in $R$ for any $f \in \Hom_\sS(\ED,\sS)$. This proves that $\xi$ is well-defined. 

It is clearly a morphism of $\sS$-modules. Checking that it respects the products is a simple exercise from their definitions as duals to the coproducts, using the commutativity of Diagram \eqref{EDeDcopr_diag} and the fact that $\aug$ is a morphism of rings. 

Finally, $\xi$ is surjective, because it sends the (dual) basis $\De_{I_w}^\star$ of $(\ED)^\star$ to the (dual) basis $\aug \De_{I_w}^*$ of $(\eD)^*$.
\end{proof}

For any $s\in \sS$, let $\ev_s \in \ED^\star$ be the map $d \mapsto d(s)$.
\begin{lem} \label{evproduct_lem}
The product on $\ED$ satisfies $\ev_{s_1} \cdot \ev_{s_2} = \ev_{s_1 s_2}$.
\end{lem}
\begin{proof}
By Theorem \ref{EDCopr_theo}, the coproduct $\EDcopr$ on $\ED$ satisfies formula \eqref{coprodprod_eq}. Thus 
\begin{equation*}
\begin{split}
(\ev_{s_1} \cdot \ev_{s_2})(d) & = (\ev_{s_1}\otimes \ev_{s_2})(\EDcopr (d)) = \Inc_2(\EDcopr(d))(s_1 \otimes s_2) \\
& = d(s_1 s_2) = \ev_{s_1s_2}(d).
\end{split}
\end{equation*}
On $\DcF^\star$, the property then follows using Theorem \ref{DF=DLambda_theo}.
\end{proof}

Let $\cmS: \sS \to \ED^\star$ be the $R$-linear map sending $s\in \sS$ to $\ev_s$. By the previous lemma, it is actually an $R$-algebra map. 
Since the subring $\sS^W\subseteq \sS \subseteq \ED$ of elements fixed by the Weyl group $W$ is central\footnote{$\ED$ is by definition generated by elements that all commute with $\sS^W$, see \cite[Prop. 3.14]{CPZ}. This is where we first use that $2$ is regular.} in $\ED$, the map $\cmS$ is in fact $\sS^W$-linear.

The map $\cmS$ would correspond in cohomology to an equivariant characteristic map. In \cite[Def. 6.1]{CPZ}, there is a (non-equivariant) characteristic map $\cmR: \sS \to (\eD)^*$, sending $s \in \sS$ to $\ev_s: \aug d \mapsto \aug d(s)$. 
They are related by $\cmR=\xi \circ \cmS$. Note that $\cmR$ is a map of $R$-algebras, either by \cite[Theorem 7.3]{CPZ} or by Lemma \ref{evproduct_lem} and the fact that $\xi$ is a morphism of $\sS$-algebras and $\sS$ acts on $\eD^*$ via $R$ through $\aug$.

The tensor product $\sS \otimes_{\sS^W} \sS$ is an $\sS$-algebra via the first component. We now consider the $\sS$-algebra map 
$$
\rhoS: \sS \otimes_{\sS^W} \sS \to \ED^\star,\qquad s_1\otimes s_2 \mapsto s_1 \cmS(s_2)
$$
and the $R$-algebra map
$$
\rhoR: R \otimes_{\sS^W} \sS \to \eD^*,\qquad r\otimes s \mapsto r \cmR(s_2).
$$
where $R$ is considered an $\sS^W$-module via the augmentation map $\aug$.
These two maps are related by the commutative diagram
$$
\xymatrix{
\sS \otimes_{\sS^W} \sS \ar[r]^{\rhoS} \ar@{->>}[d]_{\aug \otimes \id_S} & \ED^\star \ar@{->>}[d]^{\xi} \\
R \otimes_{\sS^W} \sS \ar[r]^{\rhoR} & (\eD)^* .
}
$$
The following theorem generalizes
\cite[Thm.4.4]{kk} and \cite[Thm.~11.5.15]{Ku} to the context of an
arbitrary formal group law. It is also related to \cite[Thm 5.1 and Cor. 5.2]{KiKr}, see Remark \ref{charsurjK_rem}.

\begin{theo} \label{mainthmdual_theo} 
If $2\tor$ is regular in $R$, the following conditions are equivalent. 
\begin{enumerate}
\item \label{rhoSiso_item} The map $\rhoS: S\otimes_{S^W}S\to \ED^\star$ is an isomorphism (of $\sS$-algebras).
\item \label{rhoRiso_item} The map $\rhoR: R\otimes_{S^W}S \to (\eD)^*$ is an isomorphism (of $R$-algebras). 
\item \label{charmapsurj_item} The characteristic map $\cmR$ (of $R$-algebras) is surjective. 
\end{enumerate}
In particular all are true if the torsion index $\tor$ is invertible, by \cite[Theorem 6.4]{CPZ}.
\end{theo}
\begin{proof}
By the  diagram before the theorem, if $\rhoS$ is surjective, then so is $\rhoR$. Furthermore, by definition of $\rhoR$, we have $\cmR=\rhoR\circ (1\otimes \id_S)$. Since the map $1\otimes \id_S$ is surjective (because $S^W$ surjects to $R$ via $\aug$), we have that $\cmR$ is surjective if and only if $\rhoR$ is surjective. We now need the following lemma.
\begin{lem} \label{charsurjimplies_lem}
Let $N$ be the length of $w_0$, the longest element in $W$. If $\cmR$ is surjective, then there is an element $u'_0\in \sS$ such that for any sequence $I$, we have
\[
\aug\De_I(u'_0)=\begin{cases}
0, & \text{if }I \text{ is of length }<N\text{ or is of length }N\text{
  but non reduced},\\
1, & \text{if }I \text{ is of length }N\text{ and is reduced}.
\end{cases}
\]
\end{lem}
\begin{proof}
Choose a reduced decomposition $I_w$ for every $w\in W$. By definition of $\cmR$, if it is surjective, there is an element $u'_0$ such that $\aug\Del_{I_w}(u'_0)=1$ if $w=w_0$ and zero if not. Then, decomposing $\Del_I$ as in Lemma~\ref{XIexpression_lem}, we have $\Del_I(u'_0)=\sum_v \aug(a_{I,I_v})\aug\Del_{I_v}(u'_0)=\aug(a_{I,I_{w_0}})$ and the conditions of Lemma~\ref{XIexpression_lem} imply the result.
\end{proof}
Returning to the proof of Theorem \ref{mainthmdual_theo}, let us conversely show that if $\cmR$ is surjective, then both $\rhoR$ and $\rhoS$ are isomorphisms. 
The lemma implies that the matrix $(\aug\Del_{I_v}\Del_{I_w}(u'_0))_{(v,w)\in W \times W}$ (with coefficients in $R$) is invertible, since it is triangular with $1$ on the diagonal. Thus, the matrix $(\Del_{I_v}\Del_{I_w}(u'_0))_{(v,w)\in W \times W}$ (with coefficients in $\sS$) is invertible, since the kernel of $\aug:\sS \to R$ is in the radical of $\sS$. As in \cite[Theorem 6.7]{CPZ}, this implies that the $\De_{I_w}(u'_0)$ form a basis of $\sS$ as an $\sS^W$-module. Thus, the map $\rhoS$ (resp. $\rhoR$) is between two free $\sS$-modules (resp. $R$-modules) of the same rank $|W|$, so it is an isomorphism if it is surjective. It suffices to show the surjectivity of $\rhoS$, since it implies that of $\rhoR$. The image $\rhoS(1\otimes \Del_{I_v}(u'_0))$ of a basis element of $\sS \otimes_{\sS^W} \sS$ decomposes as $\sum_w \Del_{I_w}\Del_{I_v}(u'_0) \Del_{I_w}^\star$ on the (dual) basis of $(\ED)^\star$, so by invertibility of the above matrix, we are done.
\end{proof}

\begin{theo}
If the conditions of Theorem \ref{mainthmdual_theo} hold, then $\ED =\End_{S^W}(S)$ in $\End_R(S)$.
\end{theo}
\begin{proof}
By adjunction, we have $\Hom_S(S\otimes_{S^W}S, S)\cong \End_{S^W}(S)$, mapping $f$ to $g$ such that $g(s)=f(1 \otimes s)$.
The composition 
\[
\ED\overset{\sim}{\longrightarrow} ((\ED)^\star)^\star\overset{\rhoS^\star}{\longrightarrow} \Hom_{S}(S\otimes_{S^W}S, S)\cong \End_{S^W}(S)
\]
is then precisely the inclusion $\ED \hookrightarrow \End_{S^W}(S)$, so condition \eqref{rhoSiso_item} of Theorem \ref{mainthmdual_theo} implies that $\ED = \End_{S^W}(S)$.
\end{proof}

\begin{rem} \label{charsurjK_rem}
When $F=U$ is the universal formal group law over the Lazard ring, by
\cite[Thm.~5.1]{KiKr} the algebra $\sS\ot_{\sS^W}\sS$ can be identified
(up to a completion) with the equivariant algebraic cobordism
$\Omega_T(G/B)$ after inverting $\tor$. 
Together with \cite[Cor.~5.2]{KiKr}, a result similar to Theorem
\ref{mainthmdual_theo} is shown, using topological inputs and
inverting the torsion index. Here, apart from the fact that the
treatment is completely algebraic, the main improvement is that for
other formal group laws, it might not be necessary to invert the
torsion index for the equivalent conditions of Theorem
\ref{mainthmdual_theo} to hold. For example, in the case of the
periodic multiplicative formal group law $x+y-xy$ of $K$-theory, the
characteristic map is already surjective over $\Z$ for simply
connected root data \cite{Pi}.  
\end{rem}

\begin{rem}
Note that in Lemma \ref{charsurjimplies_lem}, the element $u'_0$ has no reason to be in $\IF^N$. For example, for the periodic multiplicative formal group law $x+y-xy$ over $\Z$, for simply connected root data with non-trivial torsion index (ex: $G_2^{\sco}$), such an element $u'_0$ cannot be found in $\IF^N$.
\end{rem}

\begin{rem} 
In this section the algebra $(\ED)^\star$ can be replaced everywhere by $\DcF^\star$.
\end{rem}

\section{Appendix}

Recall that an element of a commutative, associative unital ring $R$ is called regular
if it is not a zero divisor or, in other words, that multiplication by
this element is an injective ring endomorphism.

\begin{lem} \label{zerodivisorpoly_prop}
Let $f$ be a zero divisor in the polynomial ring $R[x_1,\ldots,x_n]$. 
Then there exists an element $r\neq 0$ in $R$ such that $r f =0$. 
\end{lem}

\begin{proof}
See \cite{McCoy42}, Theorem 3 and the comment following it.
\end{proof}

Let $I=(x_1,\ldots,x_n)$ denote the ideal in the ring $R\lbr
x_1,\ldots,x_n\rbr$ of formal power series generated by $x_1,\ldots,x_n$.

\begin{cor} \label{regularpowerseries_lem}
If $f \in R\lbr x_1,\ldots, x_n\rbr$ is a zero divisor, then there
exists 
an element $r\neq 0$ in $R$ such that all coefficients of the lowest
degree part of $f$ are annihilated by $r$. 
\end{cor}

\begin{proof}
Take a nonzero $g\in R\lbr x_1, \ldots, x_n\rbr$ such that $fg=0$. Let
$f \in f_i + I^{i+1}$ and $g\in g_j + I^{j+1}$ where $f_i$ and $g_j$ are
the lowest degree parts (homogeneous and nonzero). Then $fg \in f_i g_j +I^{i+j+1}$ and,
hence, $f_i g_j=0$. The result then follows by Lemma~\ref{zerodivisorpoly_prop}. 
\end{proof}

\begin{lem} \label{regulardegree1_cor}
Let $f\in a_1 x_1 + \cdots + a_n x_n+I^2$, $a_i\in R$, be an element of $R\lbr x_1,\ldots,x_n\rbr$.
\begin{enumerate}
\item \label{onereg_item} 
If $a_i$ is regular in $R$ for some $i$, 
then $f$ is regular in $R\lbr x_1,\ldots,x_n\rbr$. 
\item \label{completebasis_item} 
If the vector of coefficients $(a_1,\ldots, a_n)$ can be completed to
a basis of $R^n$, 
then $f$ 
is regular in $R\lbr x_1,\ldots, x_n\rbr$.
\end{enumerate}
\end{lem}

\begin{proof}
Assume that $f$ is a zero divisor in $R\lbr x_1,\ldots,x_n\rbr$. Then by
Corollary~\ref{regularpowerseries_lem}
there is $r\neq 0$ in 
$R$ that annihilates each $a_i$. This contradicts the assumptions
of \eqref{onereg_item} and \eqref{completebasis_item}.
\end{proof}

Observe that there exist examples of vectors whose coordinates are all zero
divisors but that can nevertheless be completed to a basis, so that \eqref{onereg_item} does not
imply \eqref{completebasis_item}.

\medskip

Let $A=(a_{ij})$ be an $n\times n$-matrix with coefficients in $R$. 
Consider a map $\Psi_A\colon R\lbr x_1, \ldots, x_n\rbr \to R \lbr x_1,
\ldots, x_n\rbr$ 
defined by sending each $x_i$ to an element of the form $\sum a_{ij} x_j +I^2$.

\begin{lem} \label{changevarbij_prop}
The matrix $A$ is invertible if and only if the map $\Psi_A$ is invertible. 
\end{lem}

\begin{proof}
See \cite[Ch. IV, \S 4, No 7, Prop. 10]{Bo58}.
\end{proof}

\begin{cor} \label{changevarinj_prop}
If the determinant of the matrix $A$ is regular, then the map $\Psi$ is injective.
\end{cor}

\begin{proof} Let $T$ be the total ring of fractions of $R$. There is
  a canonical injection $R\hookrightarrow T$. 
The matrix $A$ viewed as a matrix with coefficients in $T$ has an
invertible determinant. Therefore, by Lemma~\ref{changevarbij_prop}
the corresponding map 
$\Psi_A\colon T\lbr x_1,\ldots, x_n \rbr \to T\lbr x_1,\ldots, x_n\rbr$ is invertible.
The injectivity of $\Psi_A$ over $R$ then follows by the
commutative diagram
\[
\xymatrix{
R\lbr x_1,\ldots,x_n\rbr \ar[r]^{\Psi_A} \ar@{^{(}->}[d] & R\lbr x_1,\ldots, x_n\rbr \ar@{^{(}->}[d]\\
T\lbr x_1,\ldots,x_n\rbr \ar[r]^{\Psi_A}_{\simeq} & T\lbr x_1,\ldots, x_n\rbr
}. \qedhere
\]
\end{proof}

By a lattice, we mean a finitely generated free Abelian group, i.e. a
free $\Z$-module of finite rank. 

\begin{lem} \label{detmatrix_lem}
Given an inclusion of lattices $\cl \subseteq \cl'$ 
of the same rank $n$, an $n\times n$ matrix $A$ expressing a basis of
$\cl$ in terms of a basis of $\cl'$ satisfies $\det(A)=\pm
|\cl'/\cl|$, the order of the group $\cl'/\cl$ up to a sign.
\end{lem}

\begin{proof}
It follows from \cite[Ch. VII, \S 4, No 6, Cor. 1]{Bo58}.
\end{proof}

An $n$-tuple of integers $(k_1,\ldots,k_n)$ is called a unimodular row
if there exist integers $(a_1,\ldots,a_n)$ such that $\sum_{i=1}^n k_i a_i =1$. 

\begin{lem} \label{unimodular_lem}
Any unimodular row $(k_1,\ldots,k_n)$ can be completed to a basis of $\Z^n$.
\end{lem}

\begin{proof}
The kernel of the map $\Z^n\to \Z$ sending $(a_1,\ldots,a_n)$ to $\sum_i k_i a_i$ is a direct factor of $\Z^n$, and it is torsion free so it is free.
\end{proof}


\begin{thebibliography}{Dem74}


\bibitem[Bo58]{Bo58}
N.~Bourbaki, 
\textit{\'El\'ements de math\'ematique. Alg\`ebre},
Hermann, Paris, 1958.

\bibitem[Bo68]{Bo68}
\bysame, 
\textit{\'El\'ements de math\'ematique. Groupes et alg\`ebres de Lie},
Hermann, Paris, 1968.

\bibitem[CPZ]{CPZ}
B.~Calm\`es, V.~Petrov and K.~Zainoulline,
\textit{Invariants, torsion indices and oriented cohomology of complete flags},
to appear in Ann. Sci. \'Ecole Norm. Sup. (4) 46, 2013.

\bibitem[CZZ]{CZZ}
B.~Calm\`es, K.~Zainoulline and C.~Zhong,
\textit{Push-pull operators on the formal affine Demazure algebra and its dual},
preprint \href{http://arxiv.org/abs/1312.0019}{arXiv.org 1312.0019}, 2013.

\bibitem[CZZ2]{CZZ2}
\bysame,
\textit{Equivariant oriented cohomology of flag varieties},
preprint \href{http://arxiv.org/abs/1409.7111}{arXiv.org 1409.7111}, 2014.

\bibitem[D73]{Dem73}
M.~Demazure,
\textit{Invariants sym\'etriques entiers des groupes de Weyl et torsion},
Invent. Math. 21:287--301, 1973.

\bibitem[De77]{Deo77}
V.~Deodhar,
\textit{Some Characterizations of Bruhat Ordering on a Coxeter Group and Determination of the Relative M\"obius Function},
Invent. Math. 39:187--198, 1977.

\bibitem[GR12]{GaRa}
N.~Ganter and A.~Ram, 
\textit{Generalized Schubert Calculus},
preprint \href{http://arxiv.org/abs/1212.5742}{arXiv.org 1212.5742}, 2012.

\bibitem[Ha78]{Haz}
M.~Hazewinkel,
\textit{ Formal groups and applications}, 
Pure and Applied Mathematics, 78. Acad. Press. New-York-London, 1978. xxii+573pp.

\bibitem[HHH]{HHH}
M.~Harada, A.~Henriques and T.~Holm, 
\textit{Computation of generalized equivariant cohomologies of Kac-Moody flag varieties},
Adv. Math. 197:198--221, 2005.

\bibitem[HMSZ]{HMSZ}
A.~Hoffmann, J.~Malag\'{o}n-L\'{o}pez, A.~Savage, and K.~Zainoulline, 
\textit{Formal Hecke algebras and algebraic oriented cohomology theories},
to appear in Selecta Math., 2013. (DOI: 10.1007/s00029-013-0132-8).

\bibitem[KiKr]{KiKr} 
V.~Kiritchenko and A.~Krishna,
\textit{Equivariant cobordism of flag varieties and of symmetric varieties},
Transf. Groups, 18:391--413, 2013.

\bibitem[KK86]{KK}
B.~Kostant and S.~Kumar,
\textit{The nil {H}ecke ring and cohomology of {$G/P$} for a {K}ac-{M}oody group {$G^*$}},
Advances in Math., 62:187--237, 1986.

\bibitem[KK90]{kk}
\bysame,
\textit{$T$-equivariant $K$-theory of generalized flag varieties},
J. Differential geometry, 32:549--603, 1990.

\bibitem[Ku02]{Ku}
S.~Kumar,
\textit{Kac-Moody groups, their flag varieties and representation theory}, 
Progress in Mathematics, vol. 204, Birkh\"auser, Boston, MA, 2002.

\bibitem[LM07]{levmor-book}
M.~Levine and F.~Morel,
\textit{Algebraic cobordism},
Springer Monographs in Mathematics. Springer-Verlag, Berlin, 2007.

\bibitem[MC42]{McCoy42}
N.H.~McCoy,
\textit{Remarks on divisors of zero},
Amer. Math. Monthly 49:286--295, 1942.  

\bibitem[Pi72]{Pi}
H.~Pittie, 
\textit{Homogeneous vector bundles on homogeneous spaces}, 
Topology 11: 199--203, 1972.

\bibitem[SGA]{SGA}
M.~Demazure and A.~Grothendieck,
\textit{Sch\'emas en groupes III: Structure des sch\'emas en groupes r\'eductifs},
(SGA 3, Vol. 3), Lecture Notes in Math. 153, Springer-Verlag, Berlin-New York, 1970, viii+529 pp. 

\bibitem[Zh]{Zh}
C.~Zhong,
\textit{On the formal affine Hecke algebra},
to appear in J. Inst. Math. Jussieu,
\href{http://dx.doi.org/10.1017/S1474748014000188}{http://dx.doi.org/10.1017/S1474748014000188}. 

\end{thebibliography}
\end{document}